\documentclass[12pt, reqno, a4paper]{amsart}

\usepackage{amssymb}
\usepackage{color}
\definecolor{darkgreen}{rgb}{0,0.45,0} 
\usepackage[pagebackref,colorlinks,citecolor=darkgreen,linkcolor=darkgreen]{hyperref}
\usepackage[matrix, arrow, curve]{xy}

\usepackage[english]{babel}
\usepackage{ulem}

\theoremstyle{plain}
\newtheorem{theorem}{Theorem}[section]
\newtheorem{lemma}[theorem]{Lemma}
\newtheorem{proposition}[theorem]{Proposition}
\newtheorem{corollary}[theorem]{Corollary}
\newtheorem{conjecture}[theorem]{Conjecture}

\theoremstyle{remark}
\newtheorem{remark}[theorem]{Remark}

\theoremstyle{remark}
\newtheorem{remarks}[theorem]{Remarks}

\theoremstyle{definition}
\newtheorem{example}[theorem]{Example}

\newtheorem{definition}[theorem]{Definition}

\numberwithin{equation}{section}

\DeclareMathOperator{\Id}{Id}
\DeclareMathOperator{\Der}{Der}
\DeclareMathOperator{\id}{id}
\DeclareMathOperator{\ch}{char}
\DeclareMathOperator{\im}{im}

\DeclareMathOperator{\End}{End}
\DeclareMathOperator{\Aut}{Aut}

\DeclareMathOperator{\supp}{supp}

\DeclareMathOperator{\Hom}{Hom}
\DeclareMathOperator{\PIexp}{PIexp}

\begin{document}

\title[Equivalences of (co)module algebra structures over Hopf algebras]{Equivalences of (co)module algebra structures over Hopf algebras}

\author{A.L. Agore}
\address{Simion Stoilow Institute of Mathematics of the Romanian Academy, P.O. Box 1-764, 014700 Bucharest, Romania}
\address{Vrije Universiteit Brussel, Pleinlaan 2, B-1050 Brussels, Belgium}
\email{ana.agore@vub.be}\thanks{Corresponding author: A.L. Agore}

\author{A.S. Gordienko}
\address{M.\,V.~Lomonosov Moscow State University,
Moscow Center for Fundamental and Applied Mathematics,
 Leninskiye gory, GSP-2, 119992 Moskva, Russia}
\address{Moscow State Technical University of Civil Aviation,
Department of Higher Mathematics, Kronshtadtsky boulevard, d.\,20, 125993 Moskva, Russia}
\email{a.gordienko@mstuca.aero}

\author{J. Vercruysse}
\address{D\'epartement de Math\'ematiques, Facult\'e des sciences, Universit\'e Libre de Bruxelles, Boulevard du Triomphe, B-1050 Bruxelles, Belgium}
\email{jvercruy@ulb.ac.be}


\begin{abstract} 
We introduce the notion of {\it support equivalence} for (co)module algebras (over Hopf algebras), which generalizes in a natural way (weak) equivalence of gradings. 
We show that for each equivalence class of (co)module algebra structures on a given algebra $A$, there exists a unique universal Hopf algebra $H$ together with an $H$-(co)module structure on $A$
such that any other equivalent (co)module algebra structure on $A$ factors through the action
of $H$.
We study support equivalence and the universal Hopf algebras mentioned above for group gradings, Hopf--Galois extensions, actions of algebraic groups and cocommutative Hopf algebras.
We show how the notion of support equivalence
can be used to reduce the classification problem of Hopf algebra (co)actions. 
We apply support equivalence in the study of the asymptotic behaviour of codimensions of $H$-identities and, in particular, to the analogue (formulated by Yu. A. Bahturin) of Amitsur's conjecture, which was originally concerned with ordinary polynomial identities. As an example we prove this analogue for all unital 
$H$-module structures on the algebra $F[x]/(x^2)$ of dual numbers.
\end{abstract}


\thanks{The first author was partially supported by a grant of Romanian Ministery of Research and Innovation, CNCS - UEFISCDI, project number PN-III-P1-1.1-TE-2016-0124, within PNCDI III and is a fellow of FWO (Fonds voor Wetenschappelijk Onderzoek~-- Vlaanderen). The second author was supported by FWO post doctoral fellowship (Belgium) and is partially supported by a grant of the scientific council of the Moscow State Technical University of Civil Aviation (Russia). The third author wants to thank the FNRS for support via the MIS ``Antipode''.}

\maketitle
\tableofcontents

\section*{Introduction}

Module and comodule algebras over Hopf algebras (see the definitions in Sections~\ref{SectionHcomodule} and~\ref{SectionHmodule} below)
 appear in many areas of mathematics and physics. These notions
allow a unified approach to algebras with various kinds of an additional structure: group gradings, group actions by automorphisms, (skew)derivations, etc.
Another important class of examples arises from (affine) algebraic geometry: If $G$ is an affine algebraic group $G$ acting morphically on an affine algebraic variety $X$, then the algebra $A$ of regular functions on $X$ is an $H$-comodule algebra where $H$ is the algebra of regular functions on $G$. 
At the same time $A$ is an $U(\mathfrak g)$-module
algebra where $U(\mathfrak g)$ is the universal enveloping algebra
of the Lie algebra $\mathfrak g$ of the algebraic group $G$ (see e.g.~\cite{Abe} and Section~\ref{SectionActionsAffAlgGroups} below.) Taking this into account, one may view a (not necessarily commutative) (co)module algebra as an action of a quantum group on a non-commutative space.

The above point of view is also advocated in \cite{Manin}, where the notion of a universal coacting Hopf algebra ${\sf aut}(A)$ on an algebra $A$ was introduced (see also \cite{Tambara}), which plays the role of a symmetry group in non-commutative geometry. In order to classify  all (co)module algebra structures on a given algebra $A$, one therefore should understand the Manin-Hopf algebra ${\sf aut}(A)$, as well as its quotients. 
For particular cases, a description of ${\sf aut}(A)$ has been obtained in, for example, \cite{RVdB} and \cite{RVdB2}. However, finding an explicit description of ${\sf aut}(A)$ seems to be a very wild problem.
Furthermore, for this universal coacting Hopf algebra of an algebra $A$ to exist, one needs to impose a finiteness condition on the algebra $A$: it needs to be finite dimensional (see \cite[Proposition 1.3.8, Remark 2.6.4]{Pareigis}) or at least a rigid object in a suitable monoidal category. An example of the latter is a graded algebra for which each homogeneous component is finite dimensional, which was the setting of \cite{Manin}.
Without such a finiteness condition, one can indeed see that ${\sf aut}(A)$ does not always exist (see \cite{AGV} for an explicit example).
Therefore, we propose here a refinement of Manin's construction by studying comodule algebras up to {\it support equivalence}  and show that a Hopf algebra coacting universally up to support equivalence exists for a given comodule algebra $A$ without any finiteness assumption on $A$.

Inspiration for our approach comes from the theory of group graded algebras.
 In this context, two gradings are called \textit{equivalent} if there exists an algebra isomorphism between the graded algebras that maps each homogeneous component onto a homogeneous component (see Section~\ref{SectionGroupGradings}). Remark that no group (iso)morphism between the grading groups is required, but only a bijection between their supports (the support of a grading group is the set of all group elements for which the corresponding homogeneous component is non-zero). 
It turns out that when one studies the structure of a graded algebra (e.g. graded ideals, graded subspaces, radicals, \dots) or graded polynomial identities of graded algebras, the grading group itself does not play an important role, but can be replaced by any other group that realizes the same decomposition of the algebra into graded components. 
Using this approach, fine gradings on exceptional simple Lie algebras have been classified up to equivalence~\cite[Chapter 6]{ElduqueKochetov}.
In~\cite{GordienkoSchnabel} the authors studied the possibility of regrading finite dimensional algebras by finite groups. 

In Theorem~\ref{TheoremGradEquivCriterion} we give a criterion for equivalence of gradings in terms of operators from the dual of the corresponding group algebra. 
As group graded algebras are exactly comodule algebras over a group algebra, we propose a generalization of the above notions of equivalence for arbitrary comodule algebras, which we call {\it support equivalence} for comodule algebras. 
Using the correspondence between $H^*$-actions and $H$-coactions for finite dimensional Hopf algebras
$H$,
 we
 also
   define support equivalence for module algebra structures and, in particular,
   for group actions. Despite the formal duality, in Proposition~\ref{PropositionDoubleNumbersUniversal} we show
   that equivalent actions of infinite dimensional Hopf algebras are in general not as close to each other
   as equivalent coactions.

Among all groups that realize a given grading there is a distinguished one called the universal group
 of the grading (see~\cite[Definition~1.17]{ElduqueKochetov}, \cite{PZ89} and Definition~\ref{def:universal} below). It is easy to see that a similar universal group exists for group actions too (Remark~\ref{RemarkUniversalGroupOfAnAction}). 
Generalizing these constructions, we show in Theorems~\ref{TheoremUniversalHcomodExistence}
 and~\ref{TheoremFinExistenceHMod} that for a given (co)module algebra $A$ not necessarily finite dimensional, there exists a unique Hopf algebra $H$ with a (co)action on $A$, which is universal among all Hopf algebras that admit a support equivalent (co)action on $A$. As mentioned before, our universal Hopf algebra provides a refinement of the Manin-Hopf algebra which is universal among {\it all} coactions on $A$ (not only those that are support equivalent).
 
We study in particular equivalences of actions of algebraic groups and their associated Lie algebras. More precisely, for a connected affine algebraic group $G$ over an algebraically closed field $F$ of characteristic $0$ and its associated Lie algebra $\mathfrak g$, we show 
that the $U(\mathfrak g)$- and the $FG$-module structures
on a finite dimensional algebra  with a rational action are support equivalent (see Theorem~\ref{TheoremAffAlgGrAllEquiv}).

Although the universal Manin--Hopf algebra of an algebra is in general very difficult to calculate, there are several interesting classes of actions for which the universal Hopf algebra defined here can be computed explicitly. 
Firstly, in case the (co)action defines a Hopf--Galois extension, the universal Hopf algebra is precisely the original Hopf algebra
 (see Theorems~\ref{TheoremHComodHopfGaloisUnivHopfAlg} and~\ref{TheoremHModHopfGaloisUnivHopfAlg}).
Similarly, the universal Hopf algebra for the standard action of a Hopf algebra $H$ on the algebra $H^*$ is precisely $H$ itself
 (Theorem~\ref{TheoremUniversalHopfOfActionOnHStarIsItself}).
Furthermore, it turns out that the universal Hopf algebra of a comodule structure
 corresponding to a grading is just the group algebra of the universal group of this grading
 (Theorem~\ref{TheoremUniversalHopfOfAGradingIsJustUniversalGroup}). Quite surprisingly, the same result does no longer hold in the case of group actions. Indeed, in Proposition~\ref{PropositionDoubleNumbersUniversal}, we show that the universal Hopf algebra of a group action can contain non-trivial primitive elements.
More generally, we prove that the universal Hopf algebra of an action of a cocommutative Hopf algebra is not necessarily cocommutative (see Example~\ref{ExampleCocommUnivDifferent}). Therefore we also consider the universal cocommutative Hopf algebras of actions. We prove their existence and describe them explicitly in the case of an algebraically closed base field of characteristic $0$, thanks to the Milnor--Moore decomposition (see Theorem~\ref{TheoremUniversalCocommutative}).
In some cases, the universal Hopf algebra coincides with the cocommutative one, and hence can be described completely. This is the case, for instance, if one can prove that the universal Hopf algebra is cocommutative too, see Proposition~\ref{PropositionDoubleNumbersUniversal}.

Finally, we apply our results to polynomial identities and show that 
the codimensions of polynomial $H_1$- and $H_2$-identities for an algebra
with support equivalent $H_1$- and $H_2$-module structures coincide
and their polynomial $H_1$- and $H_2$-non-identities can be identified in a natural way (Lemma~\ref{LemmaHEquivCodimTheSame}). 
This allows us to prove that the analog of Amitsur's conjecture holds for polynomial $H$-identities of the algebra $F[x]/(x^2)$ (see Theorem~\ref{TheoremDoubleNumbersAmitsurPIexpH}).

The paper is organized as follows. 

In Section~\ref{SectionGroupGradings} we recall the definitions of
an isomorphism and an equivalence of gradings as well as of the universal
group of a grading and give a criterion for two gradings to be equivalent in terms 
of the algebras of linear endomorphisms of the corresponding graded algebras (Theorem~\ref{TheoremGradEquivCriterion}).
In Section~\ref{SectionGroupActions} we introduce the corresponding notion
of equivalence of group actions and calculate the universal group of an action.

In Section~\ref{SectionHcomodule} we give a definition of support equivalence
of comodule structures as well as a criterion for such an equivalence
in terms of comodule maps.
In Section~\ref{SectionHcomodUniversalHopfAlgebra} we introduce
universal Hopf algebras of comodule structures and prove their existence.
In addition we show that in the case of a group grading the universal Hopf algebra of the 
corresponding comodule structure is just the group algebra of the universal group
of the grading (Theorem~\ref{TheoremUniversalHopfOfAGradingIsJustUniversalGroup})
and any coaction which is support equivalent to a grading can be always reduced to a grading
(see the precise statement in Theorem~\ref{TheoremGradingCanBeEquivalentToAGradingOnly}).
Also we show that the universal Hopf algebra can be viewed as a functor
from the preorder of all comodule structures on a given algebra (with the respect to the relation ``finer/coarser'') to the category of Hopf algebras.
In Section~\ref{SectionHcomodHopfGalois} we consider comodule Hopf--Galois extensions and show
that the corresponding universal Hopf algebra is the original coacting Hopf algebra.

In Section~\ref{SectionHmodule} we give a definition of support equivalence
of module structures. We introduce universal Hopf algebras of module structures and prove their existence in Section~\ref{univmodalg}. 
In Section~\ref{SectionHmodHopfGalois} we consider module Hopf--Galois extensions and show
that the corresponding universal Hopf algebra is the original acting Hopf algebra.

In Section~\ref{SectionActionsAffAlgGroups} we consider the classical correspondence between connected affine algebraic groups $G$ over an algebraically closed field of characteristic $0$ and their Lie algebras $\mathfrak g$. Namely, we prove that $FG$-actions and $U(\mathfrak g)$-actions are support equivalent
(Theorem~\ref{TheoremAffAlgGrAllEquiv}).
Section~\ref{SectionActionsCocommHopfAlgebras} deals with universal cocommutative Hopf algebras. The notion of a universal cocommutative Hopf algebra is then used to calculate the universal Hopf algebra (Proposition~\ref{PropositionDoubleNumbersUniversal}).
In Section~\ref{SubsectionHPI} we show how support equivalences of module structures
can be applied to classify module structures on a given algebra and to
polynomial $H$-identities.

\section{Preliminaries}
Throughout this paper, $F$ denotes a field and unless specified otherwise, all vector spaces, tensor
products, homomorphisms, (co)algebras, Hopf
algebras are over $F$. For a coalgebra
$C$, we use the Sweedler $\Sigma$-notation: $\Delta(c)=
c_{(1)}\otimes c_{(2)}$,  for all $c \in C$, with suppressed summation sign. 

Our notation for the standard categories is as follows: $\mathbf{Alg}_F$ (algebras over $F$),   $\mathbf{Coalg}_F$ (coalgebras over $F$), $\mathbf{Hopf}_F$ (Hopf algebras over $F$).
Recall that there exists a left adjoint functor $L \colon \mathbf{Coalg}_F \to \mathbf{Hopf}_F$ for the forgetful functor $U\colon \mathbf{Hopf}_F \to \mathbf{Coalg}_F$ (see~\cite{Takeuchi}) and we denote by $\eta \colon \id_{\mathbf{Coalg}_F} \Rightarrow UL$ the unit of this adjunction.
Moreover, there exists a right adjoint functor $R \colon  \mathbf{Alg}_F \to \mathbf{Hopf}_F$ for the forgetful functor $U\colon \mathbf{Hopf}_F \to \mathbf{Alg}_F$ (see~\cite{AgoreCatConstr}, \cite{CH}). The counit of this adjunction will be denoted by $\mu\colon UR \Rightarrow \id_{\mathbf{Alg}_F} $. Remark that although we use the same  character $U$ for both forgetful functors, it will be clear from the context which forgetful functor is meant.

Given a bialgebra (or a Hopf algebra) $H$, a (not necessarily associative) algebra $A$ is called an  \textit{$H$-comodule algebra} if it admits a right $H$-comodule structure $\rho \colon A \to A \otimes H$ which is an algebra homomorphism ($A \otimes H$ has the usual tensor product algebra structure). Furthermore, $A$ is called a \textit{unital} $H$-comodule algebra if there exists
an identity element $1_A \in A$ such that $\rho(1_A)=1_A\otimes 1_{H}$. The map $\rho$ is called a \textit{comodule algebra structure} on $A$ and will be written in Sweedler notation as $\rho(a)=a_{(0)}\otimes a_{(1)}$, for all $a\in A$, again with suppressed summation sign. Explicitly, the fact that $\rho$ is an algebra homomorphism reads, for all $a$, $b \in A$, as follows:
$$
(ab)_{(0)} \otimes (ab)_{(1)} = a_{(0)} b_{(0)} \otimes a_{(1)} b_{(1)}.
$$
Note that any $H$-comodule algebra map $\rho \colon A \to A \otimes H$ gives rise to an algebra homomorphism $\zeta \colon H^* \to \End_F(A)$ defined by:
$\zeta(h^*)a=h^*(a_{(1)})a_{(0)}$ for all $a\in A$ and $h^*\in H^*$.
Here $H^*$ is the vector space dual to $H$, endowed with the structure of the algebra dual to the coalgebra $H$.
Recall that if $H$ is finite dimensional, then $H^*$ is a Hopf algebra too.  

A (not necessarily associative) algebra $A$ is called an \textit{$H$-module algebra} if it admits a left $H$-module structure such that: \begin{equation}\label{EqModCompat} h(ab)=(h_{(1)}a)(h_{(2)}b) \text{
for all }a,b\in A,\ h\in H.\end{equation}
We denote by $\zeta$ the homomorphism of algebras $H \to \End_F(A)$
defined by $\zeta(h)a=ha$, for all $h\in H$ and $a\in A$, and we call it a \textit{module algebra structure} on $A$. An $H$-module algebra $A$ will be called \textit{unital} if there exists
an identity element $1_A \in A$ such that $h 1_A =\varepsilon(h)1_A$
for all $h\in H$. 

\section{Equivalences of group gradings and group actions}

\subsection{Group gradings}\label{SectionGroupGradings}

When studying graded algebras, one has to determine, when two
graded algebras can be considered ``the same'' or equivalent.

Recall that $\Gamma \colon A=\bigoplus_{g \in G} A^{(g)}$
is a \textit{grading} on a (not necessarily associative) algebra $A$ by a group $G$ if $A^{(g)}A^{(h)}\subseteq A^{(gh)}$ for all $g,h\in G$.
Then $G$ is called the \textit{grading group} of $\Gamma$.
The algebra $A$ is called \textit{graded} by $G$.

Let \begin{equation}\label{EqTwoGroupGradings}\Gamma_1 \colon A_1=\bigoplus_{g \in G_1} A_1^{(g)},\qquad \Gamma_2
\colon A_2=\bigoplus_{g \in G_2} A_2^{(g)}\end{equation} be two gradings where $G_1$
and $G_2$ are groups and $A_1$ and $A_2$ are algebras.

The most restrictive case is when we require that both grading groups
coincide:

\begin{definition}[{e.g. \cite[Definition~1.15]{ElduqueKochetov}}]
\label{DefGradedIsomorphism}
Gradings~(\ref{EqTwoGroupGradings}) are \textit{isomorphic} if $G_1=G_2$ and there exists an isomorphism $\varphi \colon A_1 \mathrel{\widetilde\to} A_2$
of algebras such that $\varphi\left(A_1^{(g)}\right)=A_2^{(g)}$
for all $g\in G_1$.
\end{definition}

In this case we say that $A_1$ and $A_2$ are \textit{graded isomorphic}.

 If one studies the graded structure of a graded algebra or its graded polynomial identities~\cite{AljaGia, AljaGiaLa,
BahtZaiGradedExp, GiaLa, ASGordienko9}, then it is not really important by elements of which group the graded components are indexed.
A replacement of the grading group leaves both graded subspaces and graded ideals graded. In the case of graded polynomial identities reindexing
the graded components leads only to renaming the variables.
 Here we come naturally to the notion of (weak) equivalence of gradings.

\begin{definition}\label{DefGradedEquivalence}
We say that gradings~(\ref{EqTwoGroupGradings}) are \textit{(weakly) equivalent}, if there
exists an isomorphism $\varphi \colon A_1 \mathrel{\widetilde\to}A_2$
of algebras such that for every $g_1\in G_1$ with $A_1^{(g_1)}\ne 0$ there
exists $g_2\in G_2$ such that $\varphi\left(A_1^{(g_1)}\right)=A_2^{(g_2)}$.
\end{definition}

Obviously, if gradings are isomorphic, then they are equivalent.
It is important to notice that the converse is not true.

However, if gradings~(\ref{EqTwoGroupGradings}) are equivalent
and $\varphi \colon A_1 \mathrel{\widetilde\to}A_2$ is the corresponding isomorphism of algebras,
then $\Gamma_3 \colon A_1=\bigoplus_{g \in G_2} \varphi^{-1}\left( A_2^{(g)}\right)$ is a $G_2$-grading on $A_1$ isomorphic
to $\Gamma_2$ and the grading $\Gamma_3$ is obtained from $\Gamma_1$ just by reindexing the homogeneous components.
Therefore, when gradings~(\ref{EqTwoGroupGradings}) are equivalent, we say that $\Gamma_1$ \textit{can be regraded} by $G_2$.

If $A_1=A_2$ and $\varphi$ in Definition~\ref{DefGradedEquivalence} is the identity map, we say that $\Gamma_1$ and $\Gamma_2$ are \textit{realizations
of the same grading} on $A$ as, respectively, a $G_1$- and an $G_2$-grading.

For a grading $\Gamma \colon A=\bigoplus_{g \in G} A^{(g)}$, we denote by $\supp \Gamma := \lbrace g\in G \mid A^{(g)}\ne 0\rbrace$ its support.

\begin{remark}Each equivalence between gradings $\Gamma_1$ and $\Gamma_2$
 induces a bijection $\supp \Gamma_1 \mathrel{\widetilde\to} \supp \Gamma_2$.
\end{remark}

Each group grading on an algebra can be realized as a $G$-grading for many different groups $G$, however it turns out that there is one distinguished group among them~\cite[Definition~1.17]{ElduqueKochetov}, \cite{PZ89}.

\begin{definition}\label{def:universal}
Let $\Gamma$ be a group grading on an algebra $A$. Suppose that $\Gamma$ admits a realization
as a $G_\Gamma$-grading for some group $G_\Gamma$. Denote by $\varkappa_\Gamma$ the corresponding embedding
$\supp \Gamma \hookrightarrow G_\Gamma$. We say that $(G_\Gamma,\varkappa_\Gamma)$ is the \textit{universal group of the grading $\Gamma$} if for any realization of $\Gamma$ as a grading by a group $G$
with $\psi \colon \supp \Gamma \hookrightarrow G$ there exists
a unique homomorphism $\varphi \colon G_\Gamma \to G$ such that the following diagram is commutative:
$$\xymatrix{ \supp \Gamma \ar[r]^(0.6){\varkappa_\Gamma} \ar[rd]_\psi & G_\Gamma \ar@{-->}[d]^\varphi \\
& G
}
$$
\end{definition}

\begin{remarks} 
a) For each grading $\Gamma$ one can define a category $\mathcal C_\Gamma$
where the objects are all pairs $(G,\psi)$ such that $G$ is a group and $\Gamma$ can be realized
as a $G$-grading with $\psi \colon \supp \Gamma \hookrightarrow G$ being the embedding of the support.
In this category the set of morphisms between $(G_1,\psi_1)$ and $(G_2,\psi_2)$ consists
of all group homomorphisms $f \colon G_1 \to G_2$ such that the diagram below is commutative:
$$\xymatrix{ \supp \Gamma \ar[r]^(0.6){\psi_1} \ar[rd]_{\psi_2} & G_1 \ar[d]^f \\
& G_2
}
$$
Then $(G_\Gamma,\varkappa_\Gamma)$ is the initial object of $\mathcal C_\Gamma$.

b) It is easy to see that if $\Gamma \colon A = \bigoplus_{g\in\supp \Gamma} A^{(g)}$ is a group grading,
then the universal group $G_\Gamma$ of the grading $\Gamma$ is isomorphic to $\mathcal F_{[\supp \Gamma]}/N$ where $\mathcal F_{[\supp \Gamma]}$ is the free group on the set $[\supp \Gamma]:=\lbrace [g] \mid g \in \supp \Gamma \rbrace$ and $N$ is the normal closure of the words $[g][h][t]^{-1}$ for pairs $g,h \in \supp \Gamma$
such that $A^{(g)}A^{(h)}\ne 0$ where $t\in \supp\Gamma$ is defined by $A^{(g)}A^{(h)}\subseteq A^{(t)}$.

c) Of course, from the linguistic point of view, it would be more logic to say in our definition ``finer or equivalent'' instead of just ``finer'' and ``coarser or equivalent'' instead of just ``coarser'', however throughout this paper we drop the words ``or equivalent'' for brevity.
\end{remarks}

Let $\Gamma_1 \colon A = \bigoplus_{g\in G} A^{(g)}$ and 
$\Gamma_2 \colon A = \bigoplus_{h\in H} A^{(h)}$ be two gradings on the same algebra $A$.
If for every $g\in G$ there exists $h\in H$ such that $A^{(g)}\subseteq A^{(h)}$, then we say
that $\Gamma_1$ is \textit{finer} than $\Gamma_2$ and 
$\Gamma_2$ is \textit{coarser} than $\Gamma_1$. It is easy to see that this relation
is a preorder and $\Gamma_1$ is both finer and coarser than $\Gamma_2$ if and only if 
$\id_A$ is an equivalence of $\Gamma_1$ and $\Gamma_2$. Moreover, the universal group
of the grading is the functor from this preorder to the category of groups:
if $\Gamma_1$ is finer than $\Gamma_2$, the functor assigns the homomorphism
$G_{\Gamma_1} \to G_{\Gamma_2}$ defined by $[g] \mapsto [h]$ for $A^{(g)}\subseteq A^{(h)}$.

Now, in order to make it possible to transfer the relation of support equivalence and the relation ``coarser/finer'' to (co)module algebra structures, we translate these relations into the language of linear operators.

If an algebra $A = \bigoplus_{g\in G} A^{(g)}$ is graded by a group $G$,
then we have an $(FG)^*$-action on $A$ where $(FG)^*$ is the algebra dual to the group coalgebra $FG$, i.e. $(FG)^*$ is the algebra of all functions $G \to F$ with pointwise operations: $ha=h(g)a$ if $g\in G$, $a\in A^{(g)}$, $h\in (FG)^*$. 

\begin{lemma}\label{LemmaGradEquivCriterion}
Let~(\ref{EqTwoGroupGradings}) be two group gradings
and let $\varphi \colon A_1 \mathrel{\widetilde\to} A_2$
be an isomorphism of algebras.
Denote by $\zeta_i \colon (FG_i)^* \to \End_{F}(A_i)$
the homomorphism from $(FG_i)^*$ to the algebra $\End_{F}(A_i)$ of $F$-linear operators on $A_i$ induced by the $(FG_i)^*$-action, $i=1,2$, and by $\tilde\varphi$  the isomorphism $\End_{F}(A_1) \mathrel{\widetilde\to} \End_{F}(A_2)$ defined by $\tilde\varphi(\psi)(a)=\varphi\Bigl(\psi\bigl(\varphi^{-1}(a)\bigr)\Bigr)$ for $\psi\in \End_{F}(A_1)$ and $a\in A_2$.
Then the inclusion \begin{equation}\label{EqEmbeddingOfImagesOfFG}
\tilde\varphi\Bigl(\zeta_1\bigl((FG_1)^*\bigr)\Bigr)\supseteq \zeta_2\bigl((FG_2)^*\bigr)
\end{equation} holds
if and only if for every $g_1 \in G_1$ there exists $g_2\in G_2$
such that $\varphi\left(A_1^{(g_1)}\right)\subseteq A_2^{(g_2)}$.
\end{lemma}
\begin{proof}
If for every $g_1 \in G_1$ there exists $g_2\in G_2$
such that $\varphi\left(A_1^{(g_1)}\right)\subseteq A_2^{(g_2)}$,
then each $A_2^{(g_2)}$ is a direct sum of some
of $\varphi\left(A_1^{(g_1)}\right)$
since $$\bigoplus_{g_1 \in G_1} \varphi\left(A_1^{(g_1)}\right)=\varphi(A_1)=A_2=\bigoplus_{g_2 \in G_2} A_2^{(g_2)}.$$
Note that the set $\zeta_i\bigl((FG_i)^*\bigr)$ 
consists of all the linear operators that act by a scalar operator
on each homogeneous component $A_i^{(g_i)}$, $g_i \in G_i$.
 Hence 
$\tilde\varphi\Bigl(\zeta_1\bigl((FG_1)^*\bigr)\Bigr)$ consists of
all the linear operators on $A_2$
that act on each $\varphi\Bigl(\zeta_1\left(A_1^{(g_1)}\right)\Bigr)$ by a scalar operator.
Since each $A_2^{(g_2)}$ is a direct sum of some
of $\varphi\left(A_1^{(g_1)}\right)$, all the operators from $\zeta_2 \bigl((FG_2)^*\bigr)$
act by a scalar operator on each of $\varphi\left(A_1^{(g_1)}\right)$ too.
Therefore~(\ref{EqEmbeddingOfImagesOfFG}) holds.

Conversely, suppose~(\ref{EqEmbeddingOfImagesOfFG}) holds.
Denote by $p_{g_2}$, where $g_2 \in \supp \Gamma_2$, the projection on $A_2^{(g_2)}$
along $\bigoplus\limits_{\substack{h\in \supp\Gamma_2, \\ h \ne g_2}} A_2^{(h)}$,
i.e. $p_{g_2} a := a$ for all $a\in A_2^{(g_2)}$
and $p_{g_2} a := 0$ for all $a \in \bigoplus\limits_{\substack{h\in \supp\Gamma_2, \\ h \ne g_2}} A_2^{(h)}$. Then~(\ref{EqEmbeddingOfImagesOfFG}) implies that $p_{g_2}\in \tilde\varphi\Bigl(\zeta_1\bigl((FG_1)^*\bigr)\Bigr)$
for all $g_2 \in \supp \Gamma_2$. In particular, $p_{g_2}$ is 
acting as a scalar operator on all the components $\varphi\left(A_1^{(g)}\right)$
where $g\in \supp\Gamma_1$. Since $p_{g_2}^2 = p_{g_2}$,
for every $g\in \supp\Gamma_1$ either $p_{g_2} \varphi\left(A_1^{(g)}\right) = 0$
 or $p_{g_2} a = a$
for all $a\in \varphi\left(A_1^{(g)}\right)$.
Hence $A_2^{(g_2)}=\im(p_{g_2})$ is a direct sum of some of $\varphi\left(A_1^{(g)}\right)$
where $g\in\supp \Gamma_1$. Since $$
\bigoplus\limits_{g_1\in \supp\Gamma_1} \varphi\left(A_1^{(g_1)}\right) =
\bigoplus\limits_{g_2\in \supp\Gamma_2} A_2^{(g_2)},$$
for every $g_1 \in G_1$ there exists $g_2\in G_2$
such that $\varphi\left(A_1^{(g_1)}\right)\subseteq A_2^{(g_2)}$.
 \end{proof}

From Lemma~\ref{LemmaGradEquivCriterion} we immediately deduce the criteria that are crucial to transfer the notion of equivalence and the relation ``finer/coarser'' from gradings to (co)module structures:

\begin{theorem}\label{TheoremGradEquivCriterion}
Let~(\ref{EqTwoGroupGradings}) be two group gradings.
Then an isomorphism $\varphi \colon A_1 \mathrel{\widetilde\to} A_2$
of algebras
defines an equivalence of gradings
if and only if 
$$
\tilde\varphi\Bigl(\zeta_1\bigl((FG_1)^*\bigr)\Bigr)=\zeta_2\bigl((FG_2)^*\bigr)$$
where $\zeta_i \colon (FG_i)^* \to \End_{F}(A_i)$
is the homomorphism from $(FG_i)^*$ to the algebra $\End_{F}(A_i)$ of $F$-linear operators on $A_i$ induced by the $(FG_i)^*$-action, $i=1,2$, and the isomorphism $\tilde\varphi \colon \End_{F}(A_1) \mathrel{\widetilde\to} \End_{F}(A_2)$ is defined by $\tilde\varphi(\psi)(a)=\varphi\Bigl(\psi\bigl(\varphi^{-1}(a)\bigr)\Bigr)$ for $\psi\in \End_{F}(A_1)$ and $a\in A_2$.
\end{theorem}
\begin{proof}
We apply Lemma~\ref{LemmaGradEquivCriterion} to $\varphi$ and $\varphi^{-1}$.
\end{proof}

\begin{theorem}\label{TheoremGradFinerCoarserCriterion}
Let $\Gamma_1 \colon A = \bigoplus_{g\in G_1} A^{(g)}$ and 
$\Gamma_2 \colon A = \bigoplus_{g\in G_2} A^{(g)}$ be two gradings on the same algebra $A$
and let $\zeta_i \colon (FG_i)^* \to \End_{F}(A)$ be the corresponding homomorphisms.
Then $\Gamma_1$ is finer than $\Gamma_2$ if and only if
$\zeta_1\bigl((FG_1)^*\bigr) \supseteq \zeta_2\bigl((FG_2)^*\bigr)$.
\end{theorem}
\begin{proof}
We apply Lemma~\ref{LemmaGradEquivCriterion} to $A=A_1=A_2$ and $\varphi = \id_A$.
\end{proof}

\subsection{Group actions}\label{SectionGroupActions}

Suppose we have a dual situation: for $i=1,2$ a group $G_i$ is acting on an algebra $A_i$ by automorphisms
where $\zeta_i \colon G_i \to \Aut(A_i)$ are the corresponding group homomorphisms.
Inspired by Theorem~\ref{TheoremGradEquivCriterion}, we introduce the following definition:

\begin{definition}\label{DefGroupActionEquivalence}
We say that actions $\zeta_1$ and $\zeta_2$ are \textit{equivalent
via an isomorphism $\varphi$} if $\varphi \colon A_1 \mathrel{\widetilde\to}A_2$
is an isomorphism
of algebras such that \begin{equation*}
\tilde\varphi\Bigl(\langle \zeta_1\bigl(G_1 \bigr) \rangle_F\Bigr)=\langle \zeta_2\bigl(G_2 \bigr) \rangle_F
\end{equation*}
where the $F$-linear span $\langle \cdot \rangle_F$ is taken in the corresponding $\End(A_i)$ and the isomorphism $\tilde\varphi \colon \End_{F}(A_1) \mathrel{\widetilde\to} \End_{F}(A_2)$ is defined by $\tilde\varphi(\psi)(a)=\varphi\Bigl(\psi\bigl(\varphi^{-1}(a)\bigr)\Bigr)$ for $\psi\in \End_{F}(A_1)$ and $a\in A_2$.
\end{definition}

If $\Gamma \colon A=\bigoplus_{g\in G} A^{(g)}$ is a grading on an algebra $A$ by a group $G$, there exists a standard $\Hom(G,F^\times)$-action on $A$ by automorphisms: \begin{equation}\label{EqGroupActionChi}\chi a := \chi(g)a\text{ for }\chi \in \Hom(G,F^\times),\ a\in A^{(g)},\ g\in G.\end{equation}
Extend each $\chi \in \Hom(G,F^\times)$ by linearity to a map
$FG \to F$. Then this $\Hom(G,F^\times)$-action becomes just the restriction of the map $\zeta \colon (FG)^* \to \End_F(A)$ corresponding to $\Gamma$.

 In the case when $F$ is algebraically closed of characteristic $0$
and $G$ is finite abelian, the group $\Hom(G,F^\times)$ is usually denoted by $\widehat G$
and called the \textit{Pontryagin dual group}. The classical theorem on the structure of finitely generated
abelian groups implies then $\widehat G \cong G$. Moreover, (\ref{EqGroupActionChi}) defines
a one-to-one correspondence between $G$-gradings and $\widehat G$-actions. (See the details
e.g. in~\cite[Section 3.2]{ZaiGia}.)
 The proposition below shows that in this case 
equivalent $G$-gradings correspond to equivalent $\widehat G$-actions.

\begin{proposition}
Let~(\ref{EqTwoGroupGradings}) be two group gradings by finite abelian groups $G_1$ and $G_2$
and let $\zeta_i$ be the corresponding 
$(FG_i)^*$-actions, $i=1,2$.
 Suppose that the base field $F$ is algebraically closed of characteristic $0$. 
Then $\Gamma_1$ and $\Gamma_2$ are equivalent as group gradings if and only if $\zeta_1 \bigr|_{\widehat G_1}$
and $\zeta_2 \bigr|_{\widehat G_2}$ are equivalent as group actions.
\end{proposition}
\begin{proof}
The orthogonality relations for characters imply that the elements of $\widehat G_i$ form a basis in $(FG_i)^*$. Hence $\zeta_i((FG_i)^*)= \langle\zeta_i(\widehat G_i )\rangle$ for $i=1,2$ and the proposition follows
from Theorem~\ref{TheoremGradEquivCriterion}.
\end{proof}

Return now to the case of arbitrary groups $G_1$ and $G_2$ and an arbitrary field $F$.
As in the case of gradings, we can identify $A_1$ and $A_2$ via $\varphi$. Then the equivalence of $\zeta_1$ and $\zeta_2$
means that the images of $G_1$ and $G_2$ generate the same subalgebra in the algebra of $F$-linear operators
on $A_1=A_2$.

\begin{definition}\label{UniversalGroupOfTheAction}
Let $\zeta \colon G \to \Aut(A)$ be an action of a group $G$ on an algebra $A$ and let $\varkappa_{\zeta} \colon G_{\zeta} \to \Aut(A)$ be an action of a group $G_\zeta$ equivalent to $\zeta$ via the identity isomorphism $\id_A$.  We say that the pair $(G_{\zeta}, \varkappa_{\zeta})$ is a \textit{universal group of the action $\zeta$}
if for any other action $\zeta_1 \colon G_1 \to \Aut(A)$ equivalent to $\zeta$ via $\id_A$
there exists a unique group homomorphism $\varphi \colon G_1 \to G_{\zeta}$ such that the following diagram is commutative:
$$\xymatrix{ \Aut(A)  & G_{\zeta} \ar[l]_(0.3){\varkappa_{\zeta}}  \\
& G_1 \ar[lu]^{\zeta_1} \ar@{-->}[u]_\varphi
}
$$
\end{definition}

\begin{remark}\label{RemarkUniversalGroupOfAnAction}
Consider the group $\mathcal U(\langle \zeta(G) \rangle_F)
\cap \Aut(A)$ where $\mathcal U(\langle \zeta(G) \rangle_F)$ is the group of invertible elements
of the algebra $\langle \zeta(G) \rangle_F \subseteq \End_F(A)$.
Since for every $G_1$ as above the image of $G_1$ in $\Aut(A)$ belongs
to  $\mathcal U(\langle \zeta_1(G_1) \rangle_F) \cap \Aut(A) = \mathcal U(\langle \zeta(G) \rangle_F)\cap \Aut(A)$,
 the universal group of the action $\zeta$
is (up to an isomorphism) the couple $(G_{\zeta}, \varkappa_{\zeta})$ where $$G_{\zeta}:= \mathcal U(\langle \zeta(G) \rangle_F)
\cap \Aut(A)$$
and $\varkappa_{\zeta}$ is the natural embedding $G_{\zeta} \subseteq \Aut(A)$.
\end{remark}

Let $\zeta_i \colon G_i \to \Aut(A)$, $i=1,2$, be two group actions on $A$ by automorphisms.
We say that $\zeta_1$ is \textit{finer} than $\zeta_2$ and $\zeta_2$ is \textit{coarser} than $\zeta_1$
if $\langle \zeta_2(G) \rangle_F \subseteq \langle \zeta_1(G) \rangle_F$.
Again, it is easy to see that this relation
is a preorder and $\zeta_1$ is both finer and coarser than $\zeta_2$ if and only if 
$\id_A$ is an equivalence of $\zeta_1$ and $\zeta_2$. Moreover, the universal group
of the action is the functor from this preorder to the category of groups:
if $\zeta_1$ is finer than $\zeta_2$, the functor assigns the embedding
$\mathcal U\bigl(\langle \zeta(G_2) \rangle_F\bigr)
\cap \Aut(A) \subseteq \mathcal U\bigl(\langle \zeta(G_1) \rangle_F\bigr)
\cap \Aut(A)$.

\begin{proposition}
Let~(\ref{EqTwoGroupGradings}) be two group gradings by finite abelian groups $G_1$ and $G_2$
and let $\zeta_i$ be the corresponding 
$(FG_i)^*$-actions, $i=1,2$.
 Suppose that the base field $F$ is algebraically closed of characteristic $0$. 
Then $\Gamma_1$ is finer than $\Gamma_2$ if and only if $\zeta_1 \bigr|_{\widehat G_1}$
is finer than $\zeta_2 \bigr|_{\widehat G_2}$.
\end{proposition}
\begin{proof}
Again, the orthogonality relations for characters imply that the elements of $\widehat G_i$ form a basis in $(FG_i)^*$. Hence $\zeta_i((FG_i)^*)= \langle\zeta_i(\widehat G_i )\rangle$ for $i=1,2$ and the proposition follows
from Theorem~\ref{TheoremGradFinerCoarserCriterion}.
\end{proof}

\section{Comodule algebras}

\subsection{Support equivalence of comodule structures on algebras}\label{SectionHcomodule}

Inspired by Theorem~\ref{TheoremGradEquivCriterion},
we give the following definition:

\begin{definition}
Let $A_i$ be (not necessarily associative) $H_i$-comodule algebras for Hopf algebras $H_i$, $i=1,2$. We say
that comodule structures on $A_1$ and $A_2$ are {\it support equivalent} via the algebra isomorphism
$\varphi \colon A_1 \mathrel{\widetilde\to} A_2$ 
%of algebras is an \textit{equivalence of comodule algebra structures} on $A_1$ and $A_2$ 
if 
\begin{equation}\label{EqImagesOfHistarCoincide}
\tilde\varphi\Bigl(\zeta_1\bigl(H_1^*\bigr)\Bigr)=\zeta_2\bigl(H_2^*\bigr)
\end{equation}
where $\zeta_i$ is the algebra homomorphism $H_i^* \to \End_{F}(A_i)$ induced by the comodule algebra structure on $A_{i}$ and the isomorphism $\tilde\varphi \colon \End_{F}(A_1) \mathrel{\widetilde\to} \End_{F}(A_2)$ is defined by $\tilde\varphi(\psi)(a)=\varphi\Bigl(\psi\bigl(\varphi^{-1}(a)\bigr)\Bigr)$ for $\psi\in \End_{F}(A_1)$ and $a\in A_2$.
%In this case we call comodule algebra structures on $A_1$ and $A_2$
%\textit{equivalent via the isomorphism $\varphi$}.

As the only equivalences we will consider in the sequel are support equivalences, we will use the term just {\it equivalence}.
\end{definition}

It is easy to see that each support equivalence of comodule algebras maps $H_1$-subcomodules to $H_2$-subcomodules.

As in the case of gradings, we can restrict our consideration to the case when $A_1=A_2$
and $\varphi$ is an identity map.

Let $A$ be an $H$-comodule algebra with a comodule map $\rho \colon A \to A \otimes H$ and the corresponding homomorphism of algebras $\zeta \colon H^* \to \End_F(A)$. Choose a basis $(a_\alpha)_\alpha$ in $A$ and let $\rho(a_\alpha)=\sum_{\beta} a_\beta \otimes h_{\beta\alpha}$ where $h_{\beta\alpha} \in H$.
Denote by $C(\rho)$ the $F$-linear span of all such $h_{\alpha\beta}$. Now since $(\rho\otimes \id_H)\rho=(\id_A \otimes 
\Delta_H)\rho$ and $(\id_A \otimes \varepsilon)\rho = \id_A$, our definition of $h_{\alpha\beta}$ implies
\begin{equation}\begin{split}\label{EqDeltahAHcomod}\Delta h_{\alpha\beta}= \sum_{\gamma}h_{\alpha\gamma} \otimes h_{\gamma\beta},\\
\varepsilon(h_{\alpha\beta}) = \left\lbrace\begin{array}{lll} 0 & \text{if} &
\alpha \ne \beta, \\
1 & \text{if} &
\alpha = \beta
\end{array}\right.\text{ for all }\alpha,\beta.
\end{split}\end{equation}
In particular,
$C(\rho)$ is a subcoalgebra of $H$. 
It is easy to see that $$\ker \zeta=C(\rho)^{\perp}:= \lbrace
\lambda\in H^* \mid \lambda(C(\rho))=0 \rbrace$$ and $\zeta(H^*)\cong C(\rho)^*$. In other words, $C(\rho)$ is the minimal subcoalgebra $C\subseteq H$
such that $\rho(A)\subseteq A \otimes C$.
\begin{definition}
Given an $H$-comodule algebra $A$ with coaction $\rho\colon A\to A\otimes H$, call the coalgebra $C(\rho)$ constructed above the {\it support coalgebra} of the coaction $\rho$.
\end{definition}
\begin{remark}
The support coalgebra was introduced by J.\,A.~Green in~\cite{GreenCoeff} 
under the name coefficient space. We prefer to use here the name ``support coalgebra'' as in the case of a grading it is exactly the linear span of the support of the grading.
\end{remark}

\begin{proposition}\label{PropositionHcomodEquivCriterion}
Let $A_i$ be $H_i$-comodule algebras for Hopf algebras $H_i$, $i=1,2$. Then an isomorphism
$\varphi \colon A_1 \mathrel{\widetilde\to} A_2$ of algebras is an equivalence of comodule algebra structures $\rho_i \colon A_i \to A_i\otimes H_i$, $i=1,2$, if and only if there exists an isomorphism $\tau \colon C(\rho_1) \mathrel{\widetilde\to} C(\rho_2)$ of coalgebras such that the following diagram is commutative:
\begin{equation}\label{EqC(rho)Equiv}
\xymatrix{ A_1 \ar[d]^{\varphi} \ar[r]^-{\rho_1} & A_1 \otimes C(\rho_1) \ar@<-1ex>[d]^{\varphi \otimes \tau} \\
A_2 \ar[r]^-{\rho_2} & A_2 \otimes C(\rho_2)}
\end{equation}
i.e.\ two comodule algebras are support equivalent if and only if they have isomorphic support coalgebras and they are isomorphic as comodules over their support coalgebra.
\end{proposition}
\begin{proof}
Suppose $\varphi$ is a support equivalence of $\rho_1$ and $\rho_2$.
Choose as above some basis $(a_\alpha)_\alpha$ in $A_1$.
Let $a'_\alpha := \varphi(a_\alpha)$,
$\rho_1(a_\alpha)=\sum_\beta a_\beta \otimes h_{\beta\alpha}$
and 
$\rho_2(a'_\alpha)=\sum_\beta a'_\beta \otimes h'_{\beta\alpha}$,
$h_{\beta\alpha}\in H_1$, $h'_{\beta\alpha}\in H_2$.

Assume that $\sum_{\alpha,\beta} t_{\beta \alpha} h_{\beta \alpha }=0$
for some $t_{\beta\alpha}\in F$ where only a finite number of $t_{\beta \alpha}$
are nonzero. Define linear functions $\tau_\alpha \colon A \to F$
by $\tau_\alpha(a_\beta)= t_{\beta \alpha}$.
Then for any $\lambda \in H_1^*$
we have $$\sum_\alpha \tau_\alpha(\zeta_1(\lambda)a_\alpha) = \sum_{\alpha,\beta} 
t_{\beta \alpha} \lambda(h_{\beta \alpha })=0.$$
Hence $$\sum_\alpha \tau_\alpha\Bigl(\varphi^{-1}\bigl(\tilde\varphi\left(\zeta_1(\lambda)\right)a'_\alpha\bigr)\Bigr) =0
\text{ for all }\lambda \in H_1^*.$$ Now
(\ref{EqImagesOfHistarCoincide}) implies
$$\sum_\alpha \tau_\alpha\Bigl(\varphi^{-1}\bigl(\zeta_2(\lambda')a'_\alpha\bigr)\Bigr) =0$$
and $$\sum_{\alpha,\beta} 
t_{\beta \alpha} \lambda'(h'_{\beta \alpha })=0
\text{ for all }\lambda' \in H_2^*.$$
As a consequence, $\sum_{\alpha,\beta} t_{\beta \alpha} h'_{\beta \alpha }=0$.
Applying the same argument for $\varphi^{-1}$,
we obtain that if $\sum_{\alpha,\beta} t_{\beta \alpha} h'_{\beta \alpha }=0$
for some $t_{\beta\alpha}\in F$, then $\sum_{\alpha,\beta} t_{\beta \alpha} h_{\beta \alpha }=0$.
Hence there are the same linear dependencies among $h_{\alpha\beta}$
and among $h'_{\alpha\beta}$.
Taking into account that $C(\rho_1)=\langle h_{\alpha\beta} \mid \alpha,\beta \rangle_F$
and $C(\rho_2)=\langle h'_{\alpha\beta} \mid \alpha,\beta \rangle_F$,
we can correctly define a linear map 
 $\tau \colon C(\rho_1) \to C(\rho_2)$ by $\tau(h_{\alpha\beta})=h'_{\alpha\beta}$ for all $\alpha,\beta$
 and the reverse map $\tau^{-1} \colon C(\rho_2) \to C(\rho_1)$ by $\tau^{-1}(h'_{\alpha\beta})=h_{\alpha\beta}$ for all $\alpha,\beta$. In particular, the map $\tau$ is a bijection.
By~(\ref{EqDeltahAHcomod}) the map $\tau$ is a homomorphism of coalgebras.
Finally, (\ref{EqC(rho)Equiv}) holds.

Conversely, suppose~(\ref{EqC(rho)Equiv}) holds. Then $\zeta_i(\lambda)$
for $\lambda\in H_i^*$ is determined by the values 
of $\lambda$ on $C(\rho_i)$. Now~(\ref{EqC(rho)Equiv}) implies~(\ref{EqImagesOfHistarCoincide}).
\end{proof}

\begin{remark}
The proof of Proposition~\ref{PropositionHcomodEquivCriterion} makes in fact no use of the algebra structures of $A_i$ and $H_i$, nor of the coalgebra structures of $H_i$. In fact, it is possible to define a notion of support equivalence for 
arbitrary linear maps $A \to B \otimes Q$ where $A,B,Q$ are just vector spaces, and in such a setting Proposition~\ref{PropositionHcomodEquivCriterion} remains valid mutatis mutandis (see \cite[Proposition 2.6]{ASGordienko21ALAgoreJVercruysse}).
\end{remark}

\subsection{Universal Hopf algebra of a comodule algebra structure}
\label{SectionHcomodUniversalHopfAlgebra}

Analogously to the universal group of a grading, we would like to introduce
the universal Hopf algebra of a given comodule algebra structure.
It will be the initial object of the category $\mathcal C_A^H$
defined below.

Let $A$ be an $H$-comodule algebra for a Hopf algebra $H$ and let $\zeta \colon H^* \to \End_F(A)$ be the corresponding algebra homomorphism.
Consider the category $\mathcal C_A^H$ where 
\begin{enumerate}
\item the objects are
$H_1$-comodule algebra structures on the algebra $A$ for arbitrary
Hopf algebras $H_1$ over $F$ such that $\zeta_1(H_1^*)=\zeta(H^*)$
where $\zeta_1 \colon H_1^* \to \End_F(A)$
is the algebra homomorphism corresponding to the $H_1$-comodule algebra structure on $A$;
\item the morphisms from an $H_1$-comodule algebra structure on $A$ 
with the corresponding homomorphism $\zeta_1$
to an $H_2$-comodule algebra structure with the corresponding homomorphism $\zeta_2$
are all Hopf algebra homomorphisms $\tau \colon H_1 \to H_2$
such that the following diagram is commutative:
\begin{equation}\label{EqDiagramCatAComod}\xymatrix{ \End_{F}(A)   & H_2^* \ar[l]_(0.3){\zeta_2}\ar[d]^{\tau^*} \\
& H_1^*\ar[lu]^{\zeta_1}
}
\end{equation}
\end{enumerate}
\begin{remark}
The commutative diagram~(\ref{EqDiagramCatAComod}) means that $$\zeta_2(h^*)(a)=\zeta_1(\tau^*(h^*))(a),$$
$$h^*(a_{[1]})a_{[0]}=h^*(\tau(a_{(1)}))a_{(0)}$$ for all $h^*\in H_2$
and $a\in A$ where $\rho_1(a)=a_{(0)}\otimes a_{(1)}$ and $\rho_2(a)=a_{[0]}\otimes a_{[1]}$ are the $H_1$- and $H_2$-comodule maps on $A$, respectively.
Hence $$a_{[0]} \otimes a_{[1]}=a_{(0)} \otimes \tau(a_{(1)})$$
for all $a\in A$
and~(\ref{EqDiagramCatAComod}) is equivalent to the diagram below:
\begin{equation*}\xymatrix{ A  \ar[r]^(0.4){\rho_1} \ar[rd]_{\rho_2}& A \otimes H_1 \ar[d]^{\id_A \otimes \tau} \\
& A\otimes H_2
}
\end{equation*}

\end{remark}

\medskip

Now we are going to show that $\mathcal C_A^H$ possesses an initial object
(which is always unique up to an isomorphism).

Recall that by $L$ we denote the left adjoint functor to the forgetful functor $U\colon \mathbf{Hopf}_F \to \mathbf{Coalg}_F$.
Let $\rho \colon A \to A \otimes H$
be the map defining a right $H$-comodule structure on $A$. 
Since $C(\rho) \subseteq H$ is an embedding, $\eta_{C(\rho)} \colon C(\rho) \to L(C(\rho))$
is an embedding too. Hence for our choice of $C(\rho)$ the algebra $A$ is a right $C(\rho)$-comodule and therefore a right $L(C(\rho))$-comodule. Now we will factor $L(C(\rho))$ by a specific Hopf ideal $I$ to turn $A$ into an $L(C(\rho))/I$-comodule algebra.

Choose a basis $(a_\alpha)_\alpha$ in $A$ and $h_{\alpha\beta}$
as in the previous section. 

\begin{lemma}
Let $a_\alpha a_\beta = \sum_{v} k_{\alpha\beta}^v a_v$
for some structure constants $k_{\alpha\beta}^v \in F$ and
denote by $I_0$ the ideal of $L(C(\rho))$ generated by
\begin{equation}\label{EqDefRelIUnivComod}\sum_{r,q} k^\gamma_{rq} \eta_{C(\rho)}(h_{r\alpha})\eta_{C(\rho)}(h_{q\beta}) - \sum_{u} k^{u}_{\alpha\beta}\eta_{C(\rho)} (h_{\gamma u})\end{equation} for all possible choices of indices $\alpha,\beta,\gamma$. Then $I_{0}$ is a coideal.
\end{lemma}

\begin{proof} First of all, note that for every $\alpha,\beta,\gamma$
we have $$\varepsilon\left(\sum_{r,q} k^\gamma_{rq} \eta_{C(\rho)}(h_{r\alpha})\eta_{C(\rho)}(h_{q\beta}) - \sum_{u} k^{u}_{\alpha\beta}\eta_{C(\rho)} (h_{\gamma u})
\right)=k^\gamma_{\alpha\beta}-k^\gamma_{\alpha\beta}=0$$
and therefore $\varepsilon(I_{0})=0$.
Moreover, a direct computation using \eqref{EqDeltahAHcomod} gives
\begin{equation*} \begin{split}
&\Delta \biggl(\sum_{r,q} k^\gamma_{rq} \eta_{C(\rho)}(h_{r\alpha})\eta_{C(\rho)}(h_{q\beta}) - \sum_{u} k^{u}_{\alpha\beta}\eta_{C(\rho)} (h_{\gamma u})\biggl)\\
& \stackrel{\eqref{EqDeltahAHcomod}} {=} \sum_{r,q,a,b} k^\gamma_{rq}\, \eta_{C(\rho)}(h_{ra}) \eta_{C(\rho)}(h_{qb}) \otimes \eta_{C(\rho)}(h_{a \alpha}) \eta_{C(\rho)} (h_{b \beta}) - \sum_{u,v} k^{u}_{\alpha\beta}\, \eta_{C(\rho)}(h_{\gamma v}) \otimes \eta_{C(\rho)}(h_{vu})\\
&= \sum_{a,b} \biggl(\underline{\sum_{r,q} k^\gamma_{rq} \eta_{C(\rho)}(h_{ra}) \eta_{C(\rho)}(h_{qb}) - \sum_{v} k^{v}_{ab}  \eta_{C(\rho)} (h_{\gamma v})}\biggr) \otimes \eta_{C(\rho)}(h_{a \alpha}) \eta_{C(\rho)} (h_{b \beta}) +\\
& \quad \sum_{v} \Biggl( \eta_{C(\rho)} (h_{\gamma v}) \otimes \biggl(\underline{\sum_{a,b} k^{v}_{ab} \eta_{C(\rho)}(h_{a \alpha}) \eta_{C(\rho)} (h_{b \beta}) - \sum_{u} k^{u}_{\alpha \beta} \eta_{C(\rho)}(h_{v u})} \biggr) \Biggr)
\end{split}\end{equation*}
and obviously the underlined terms belong to $I_{0}$, as desired.
\end{proof}

Now define $I$
to be the ideal generated by spaces $S^n I_0$
for all $n\in\mathbb N$ where $S$ is the antipode of $L(C(\rho))$.
Obviously, $I$ is a Hopf ideal and $H^{\rho}:=L(C(\rho))/I$
is a Hopf algebra.

Denote by $\bar \eta \colon C(\rho) \to L(C(\rho))/I$ the map induced by $\eta_{C(\rho)}$
and define an $H^{\rho}$-comodule algebra structure $\varkappa^\rho$ on $A$ 
by $\varkappa^\rho(a_\alpha):= \sum_\beta a_\beta \otimes \bar\eta(h_{\beta\alpha})$.
The relations~\eqref{EqDefRelIUnivComod} ensure that $A$ indeed becomes
an $H^{\rho}$-comodule algebra.

\begin{theorem}\label{TheoremUniversalHcomodExistence}
The pair $(H^\rho,\varkappa^\rho)$
is the initial object of the category $\mathcal C_A^H$.
\end{theorem}
\begin{proof}
We first notice that the embedding $C(\rho) \hookrightarrow H$
induces a homomorphism of Hopf algebras $\varphi
\colon L(C(\rho)) \to H$ 
such that the diagram
\begin{equation*}\xymatrix{ C(\rho)  \ar[r]^(0.45){\eta_{C(\rho)}} \ar[rd]^{}& L(C(\rho)) \ar[d]^{\varphi} \\
& H
}
\end{equation*}
is commutative. Since $A$ is an $H$-comodule algebra,
the generators~(\ref{EqDefRelIUnivComod}) of $I$ belong to the kernel of $\varphi$
and there exists a homomorphism of Hopf algebras $\bar\varphi
\colon L(C(\rho))/I \to H$ such that the diagram
\begin{equation*}\xymatrix{ C(\rho)  \ar[r]^(0.4){\bar\eta} \ar[rd]^{}& L(C(\rho))/I \ar[d]^{\bar\varphi} \\
& H
}
\end{equation*}
is commutative. 
Hence the map $\bar\eta$ is injective and therefore $\bar\eta \colon C(\rho) \to \bar\eta(C(\rho))$ is a coalgebra isomorphism. Since $\bar\eta(C(\rho))$ is the support coalgebra
of $\varkappa^\rho$, Proposition~\ref{PropositionHcomodEquivCriterion} implies that the structure of an
$H^\rho$-comodule algebra on $A$ defined by $\varkappa^\rho$
belongs to $\mathcal C_A^H$.

Suppose now that $A$ is an $H_1$-comodule algebra
for some other Hopf algebra $H_1$
and the corresponding comodule structure $\rho_1 \colon A_1 \to A_1 \otimes H_1$ is equivalent to $\rho$.
Then by Proposition~\ref{PropositionHcomodEquivCriterion}
there exists an isomorphism $\tau_0 \colon C(\rho) \mathrel{\widetilde\to}
C(\rho_1)$ such that the following diagram is commutative:
\begin{equation*} %\label{EqC(rho)EquivATheSame}
\xymatrix{ A \ar[rd]_-{\rho_1} \ar[r]^-{\rho} & A \otimes C(\rho) \ar@<-1ex>[d]^{\id_A \otimes \tau_0} \\
 & A \otimes C(\rho_1)}
\end{equation*}

Let $i_1$ be the embedding $C(\rho_1)\subseteq H_1$.
Then $i_1\tau_0 = \tau_1 \eta_{C(\rho)}$ for a unique homomorphism of Hopf algebras $\tau_1 \colon L(C(\rho))\to H_1$.
Note that since $A$ is an $H_1$-comodule algebra, again all generators~(\ref{EqDefRelIUnivComod}) of the ideal $I$ are in the kernel. Hence we get a homomorphism of Hopf algebras $\tau \colon L(C(\rho))/I\to H_1$ providing the desired arrow in $\mathcal C_A^{H}$.
This arrow is unique since, by Takeuchi's construction of the functor $L$ (see \cite{Takeuchi}), $L(C(\rho))/I$ is generated as an algebra by $\bar\eta(h_{\alpha\beta})$
and their images under $S$.
\end{proof}

We call the pair $\bigl(H^\rho, \varkappa^\rho\bigr)$ the \textit{universal Hopf algebra} of $\rho$.

\begin{remark}
If $\rho \colon A \to A\otimes H$ is a right $H$-comodule algebra structure on $A$ such that $C(\rho)$ is a pointed coalgebra then the corresponding universal Hopf algebra of $\rho$ is pointed as well. Indeed, it follows from \cite[Proposition 3.3]{Zhuang} that $L(C(\rho))$ is pointed and its coradical is given by $L(C(\rho))_{0} = F T$ where $T=\mathcal F_{G(C(\rho))}$ is the free group generated by the set $G(C(\rho))$ of group-like  elements of $C(\rho)$. Now since the canonical projection $\pi \colon L(C(\rho)) \to L(C(\rho))/I$ is a surjective coalgebra homomorphism it follows from \cite[Exercise 5.5.2]{Danara} or \cite[Corollary 5.3.5]{Montgomery} that $L(C(\rho))/I$ is pointed and its coradical is $\pi\bigl(F T\bigl)$. 
\end{remark}

\begin{theorem}\label{TheoremHopfUnivComodUnitality}
Suppose $A$ is a unital $H$-comodule algebra. Then $A$ is a unital $H^\rho$-comodule coalgebra too.
As a consequence, unital comodule structures can be equivalent only to
unital comodule structures.
\end{theorem}
\begin{proof}
We can include $1_A$ into a basis, say, $a_1 := 1_A$. We have
$$h_{\alpha 1} = \left\lbrace\begin{array}{lll} 0 & \text{if} &
\alpha \ne 1, \\
1_H & \text{if} &
\alpha = 1
\end{array}\right. $$
and
$$k^{\beta}_{\alpha 1}=k^{\beta}_{1\alpha} = \left\lbrace\begin{array}{lll} 0 & \text{if} &
\alpha \ne \beta, \\
1_F & \text{if} &
\alpha = \beta
\end{array}\right. $$
for all $\alpha,\beta$. Moreover, $\varepsilon(h_{11})=1$
and $\Delta h_{11}=h_{11}\otimes h_{11}$.
Now~(\ref{EqDefRelIUnivComod})
for $\alpha = 1$ implies $\bar\eta(h_{11})\bar\eta(h_{\gamma\beta})=\bar\eta(h_{\gamma\beta})$ and for $\beta=1$ implies
$\bar\eta(h_{\gamma\alpha})\bar\eta(h_{11})=\bar\eta(h_{\gamma\alpha})$.

Note that $$S( \bar\eta(h_{11})) \bar\eta(h_{11})=  \bar\eta(h_{11})S( \bar\eta(h_{11})) = 1_{H^\rho}$$ since
 $\Delta \bar\eta(h_{11})=\bar\eta(h_{11})\otimes\bar\eta(h_{11})$.
 
Using induction on $k$ (the base $k=0$ has been already proven), 
we get
\begin{equation*}\begin{split}S^k(\bar\eta(h_{\gamma\beta}))
\bar\eta(h_{11}) = S\left( \bar\eta(h_{11}) S^{k-1}(\bar\eta(h_{\gamma\beta})) \right)
\bar\eta(h_{11}) \\ = S^k(\bar\eta(h_{\gamma\beta}))S( \bar\eta(h_{11})) \bar\eta(h_{11}) = S^k(\bar\eta(h_{\gamma\beta})).\end{split}\end{equation*}
The equality $\bar\eta(h_{11}) S^k(\bar\eta(h_{\gamma\beta})) = S^k(\bar\eta(h_{\gamma\beta}))$
is proven analogously.
In other words, $\bar\eta(h_{11})$ is the identity element of 
$H^\rho$ and $A$ is a unital $H^\rho$-comodule algebra.
\end{proof}

Theorem~\ref{TheoremUniversalHopfOfAGradingIsJustUniversalGroup}
below shows that in the case of gradings the construction above
yields the group algebra of the universal group of the corresponding grading.

\begin{theorem}\label{TheoremUniversalHopfOfAGradingIsJustUniversalGroup}
Let $\Gamma \colon A = \bigoplus_{g\in G} A^{(g)}$ be a grading on an algebra $A$ by a group $G$. Denote by $\rho \colon A \to A \otimes FG$ the corresponding comodule map.
Let $G_\Gamma$ be the universal group of $\Gamma$ and let $\rho_\Gamma \colon A \to A \otimes FG_\Gamma$
be the corresponding comodule map.
Then $(FG_\Gamma, \rho_\Gamma)$ is the universal Hopf algebra of the comodule structure $\rho$.
\end{theorem}
\begin{proof}
If $H$ is the group algebra $FG$ for some group $G$, i.e. $A$ is $G$-graded,
then we can choose $(a_\alpha)_\alpha$ to be a homogeneous basis
in $A$. In this case $h_{\alpha\beta}=0$ for $\alpha\ne \beta$
and each $h_{\alpha\alpha}$ is the group element corresponding to the homogeneous component of $a_\alpha$. The generators~(\ref{EqDefRelIUnivComod}) defining $I$ now correspond to the relations defining the universal group of the grading. From Takeuchi's construction~\cite{Takeuchi} it follows
that $L(C(\rho))$ is the group algebra of the free group generated by the support of the grading and $H^\rho=L(C(\rho))/I$ is the group algebra of 
the universal group of the grading.
\end{proof}

\begin{theorem}\label{TheoremGradingCanBeEquivalentToAGradingOnly}
If $\rho \colon A \to A \otimes H$
is an $H$-comodule structure equivalent to a group grading,
then there exists a Hopf subalgebra $H_1 \subseteq H$, isomorphic
to a group Hopf algebra, such that $\rho(A)\subseteq A \otimes H_1$.
\end{theorem}
\begin{proof}
In order to deduce this from Theorem~\ref{TheoremUniversalHopfOfAGradingIsJustUniversalGroup},
it is sufficient to consider
the homomorphic image $H_1$ of $L(C(\rho))/I$
and use the fact that any homomorphism of Hopf algebras maps group-like elements to group-like elements.
\end{proof}

Next we prove that taking the universal Hopf algebra of a comodule algebra yields a functor. To start with, given an algebra $A$, we define the category $\mathcal C_{A}$  as follows:
\begin{enumerate}
\item the objects are pairs $(H,\, \rho)$ where $H$ is a Hopf algebra and $\rho \colon A \to A \otimes H$ is a right $H$-comodule algebra structure on $A$;

\item the morphisms between two objects $(H_{1},\, \rho_{1})$ and $(H_{2},\, \rho_{2})$ are coalgebra homomorphisms $\tau \colon C(\rho_{1}) \to C(\rho_{2})$ such that the following diagram is commutative:
\begin{equation}\label{functor}
\xymatrix{ A  \ar[r]^(0.4){\rho_{1}} \ar[rd]_{\rho_{2}}& A \otimes C(\rho_{1}) \ar[d]^{\id_A \otimes \tau} \\
& A\otimes C(\rho_{2})
}
\end{equation}
\end{enumerate}

\begin{theorem} If $\rho_i \colon A \to A \otimes H_i$, $i=1,2$, are comodule structures 
on $A$ and $\zeta_i \colon H_i^* \to \End_F(A)$ are the corresponding $H_i^*$-actions, then in $\mathcal C_{A}$ there exists at most one morphism from $(H_1,\, \rho_1)$
to $(H_2,\, \rho_2)$. Furthermore, this morphism exists if and only if $\zeta_2(H_2^*) \subseteq \zeta_1(H_1^*)$. In particular, $\mathcal C_{A}$ is a preorder (in the sense of~\cite[Chapter I, Section 2]{MacLaneCatWork})\footnote{Or, in another terminology, a thin category.}.
\end{theorem}
\begin{proof}
As before, choose a basis $(a_\alpha)_\alpha$ in $A$ and let $\rho_1(a_\alpha)=\sum_{\beta} a_\beta \otimes h_{\beta\alpha}$ where $h_{\beta\alpha} \in H_1$ 
and $\rho_2(a_\alpha)=\sum_{\beta} a_\beta \otimes h'_{\beta\alpha}$
where $h'_{\beta\alpha} \in H_2$.

The proof of Proposition~\ref{PropositionHcomodEquivCriterion}
implies that if $\zeta_2(H_2^*) \subseteq \zeta_1(H_1^*)$, then if among $h_{\beta\alpha}$ there exist some
linear dependence, the same linear dependence holds among $h'_{\beta\alpha}$.
Therefore there exists a linear map $\tau \colon C(\rho_1) \to C(\rho_2)$
such that $\tau(h_{\beta\alpha})=h'_{\beta\alpha}$, making~\eqref{functor} commutative.
By~\eqref{EqDeltahAHcomod} $\tau$ is a coalgebra homomorphism.

Conversely, suppose~\eqref{functor} is commutative for some coalgebra homomorphism $\tau \colon C(\rho_1) \to C(\rho_2)$. Then $\tau(h_{\beta\alpha})=h'_{\beta\alpha}$ for all $\alpha$ and $\beta$.
Since $C(\rho_1)$ is the $F$-linear span of all $h_{\beta\alpha}$, such $\tau$ is unique
and there exists at most one morphism from $(H_1,\, \rho_1)$
to $(H_2,\, \rho_2)$, which is always surjective.
For every $f \in \zeta_2(H_2^*)$
there exists $\alpha \in C(\rho_2)^*$ such that
 we have $fa=(\id_A \otimes \alpha)\rho_2(a)$
for all $a\in A$.
The commutativity of~\eqref{functor} implies 
$fa=(\id_A \otimes \tau^*(\alpha))\rho_1(a)=\tau^*(\alpha)a$
where $\tau^*(\alpha) \in C(\rho_1)^*$.
In other words, $f \in \zeta_1(H_1^*)$
and $\zeta_2(H_2^*) \subseteq \zeta_1(H_1^*)$.
\end{proof}

If $\zeta_2(H_2^*) \subseteq \zeta_1(H_1^*)$, we say that $\rho_1$ is \textit{finer} than $\rho_2$
and $\rho_2$ is \textit{coarser} than $\rho_1$.
Note that Theorem~\ref{TheoremGradFinerCoarserCriterion} implies that this definition
agrees with the one for gradings. Again, $\id_A$ is an equivalence of $\rho_1$ and $\rho_2$
if and only if $\rho_1$ is both finer and coarser than $\rho_2$.
Furthermore, the proof of Theorem~\ref{TheoremUniversalHcomodExistence}
implies that the universal Hopf algebra of a comodule structure $\rho$ is universal not only among the structures equivalent to $\rho$, but also among all the structures that are coarser than $\rho$.

We claim that any morphism $\tau \colon (H_{1},\, \rho_{1}) \to (H_{2},\, \rho_{2})$ in $\mathcal C_{A}$ induces a Hopf algebra homomorphism between the corresponding universal Hopf algebras $L(C(\rho_{1}))/I_{1}$ and $L(C(\rho_{2}))/I_{2}$, respectively.

 To this end, consider $(a_{\alpha})$ to be a basis of $A$ over $F$ and let $C(\rho_{1})$, $C(\rho_{2})$ be the $F$-spans of all $h_{\alpha \beta}$, respectively $t_{\alpha \beta}$, where $\rho_{1}(a_{\alpha}) = \sum_{\beta} a_{\beta} \otimes h_{\beta \alpha}$ and $\rho_{2}(a_{\alpha}) = \sum_{\beta} a_{\beta} \otimes t_{\beta \alpha}$. Then, the coalgebra homomorphism $\tau \colon C(\rho_{1}) \to C(\rho_{2})$ determines a unique Hopf algebra homomorphism $\varphi \colon L(C(\rho_{1})) \to L(C(\rho_{2}))$ such that the following diagram is commutative:
\begin{equation*}\xymatrix{ C(\rho_{1})  \ar[r]^(0.45){\eta_{C(\rho_{1})}} \ar[d]_{\tau}& L(C(\rho_{1})) \ar[d]^{\varphi} \\
C(\rho_{2}) \ar[r]_(0.45){\eta_{C(\rho_{2})}} & L(C(\rho_{2}))
} \qquad {\rm
i.e.} \,\, \varphi \eta_{C(\rho_{1})} = \eta_{C(\rho_{2})} \tau.
\end{equation*}
Let $\pi_{1} \colon L(C(\rho_{1})) \to L(C(\rho_{1}))/I_{1}$ and $\pi_{2} \colon L(C(\rho_{2})) \to L(C(\rho_{2}))/I_{2}$ be the canonical projections. Now notice that by the commutativity of ~(\ref{functor}), we obtain $\tau(h_{\beta \alpha}) = t_{\beta \alpha}$ for all $\alpha$, $\beta$ which together with the commutativity of the above diagram implies by a straightforward computation that all generators ~(\ref{EqDefRelIUnivComod}) of $I_{1}$ belong to the kernel of $\pi_{2}  \varphi$. Therefore, there exists a unique Hopf algebra homomorphism $\overline{\tau} \colon L(C(\rho_{1}))/I_{1} \to L(C(\rho_{2}))/I_{2}$ such that the following diagram is commutative:
\begin{equation}\label{hopfmap}\xymatrix{ L(C(\rho_{1}))  \ar[d]_(0.45){\pi_{1}} \ar[r]^(0.45){\pi_{2}  \varphi} & L(C(\rho_{2}))/I_{2} \\
L(C(\rho_{1}))/I_{1} \ar[ur]_(0.45){\overline{\tau}}  & {} 
}  \qquad {\rm
i.e.} \,\, \overline{\tau}  \pi_{1} = \pi_{2}  \varphi.
\end{equation}

We have in fact defined a functor from $\mathcal C_{A}$ to the category of Hopf algebras:

\begin{theorem}\label{functorial1}
There exists a functor $F \colon \mathcal C_{A} \to \mathbf{Hopf}_F$ given as follows:
$$F(H, \rho) = L(C(\rho))/I\,\, {\rm and}\,\, F(\tau) = \overline{\tau}.$$
\end{theorem}
\begin{proof}
We only need to show that $F$ respects composition of morphisms. Indeed, consider $\tau_{1} \colon (H,\, \rho) \to (H_{1},\, \rho_{1})$ and $\tau_{2} \colon (H_{1},\, \rho_{1}) \to (H_{2},\, \rho_{2})$ two morphisms in $\mathcal C_{A}$. We obtain two unique Hopf algebra homomorphisms $\varphi_{1} \colon L(C(\rho)) \to L(C(\rho_{1}))$ and $\varphi_{2} \colon L(C(\rho_{1})) \to L(C(\rho_{2}))$ such that
\begin{equation}\label{FC1}
\varphi_{1} \eta_{C(\rho)} = \eta_{C(\rho_{1})}  \tau_{1}
\end{equation}
\begin{equation}\label{FC2}
\varphi_{2} \eta_{C(\rho_{1})} = \eta_{C(\rho_{2})}  \tau_{2}
\end{equation}
Furthermore, we have unique Hopf algebra homomorphisms $\overline{\tau_{1}} \colon L(C(\rho))/I \to L(C(\rho_{1}))/I_{1}$ and $\overline{\tau_{2}} \colon L(C(\rho_{1}))/I_{1} \to L(C(\rho_{2}))/I_{2}$ such that
\begin{equation}\label{FC3}
\overline{\tau_{1}} \pi = \pi_{1}  \varphi_{1}
\end{equation}
\begin{equation}\label{FC4}
\overline{\tau_{2}} \pi_{1} = \pi_{2}  \varphi_{2}
\end{equation}  
Similarly, there exist two unique Hopf algebra homomorphisms $\psi \colon L(C(\rho)) \to L(C(\rho_{2}))$ and $\overline{\tau_{2} \tau_{1}} \colon L(C(\rho))/I \to L(C(\rho_{2}))/I_{2}$ such that
\begin{equation}\label{FC5}
\psi \,\eta_{C(\rho)} = \eta_{C(\rho_{2})} \tau_{2} \tau_{1}
\end{equation}
\begin{equation}\label{FC6}
\overline{\tau_{2} \tau_{1}} \pi  = \pi_{2}  \psi
\end{equation}  
The proof will be finished once we show that $\overline{\tau_{2} \tau_{1}} = \overline{\tau_{2}}\,  \overline{\tau_{1}}$. First we prove that $\psi =  \varphi_{2}  \varphi_{1}$. 
Indeed, using~(\ref{FC1}) and~(\ref{FC2}) we obtain
\begin{equation*}
\varphi_{2} \, \varphi_{1} \eta_{C(\rho)}= \varphi_{2} \eta_{C(\rho_{1})} \tau_{1} =  \eta_{C(\rho_{2})} \tau_{2} \tau_{1}  
\end{equation*}
As $\psi$ is the unique Hopf algebra homomorphism for which~(\ref{FC5}) holds we obtain $\psi = \varphi_{2}  \varphi_{1}$. Furthermore, using~(\ref{FC4}) and~(\ref{FC3}) yields
\begin{equation*}
\pi_{2} \varphi_{2}  \varphi_{1} = \overline{\tau_{2}} \pi_{1} \varphi_{1} = \overline{\tau_{2}}\,  \overline{\tau_{1}} \pi 
\end{equation*}
Finally, the uniqueness of the Hopf algebra homomorphism for which~(\ref{FC6}) holds implies $\overline{\tau_{2} \tau_{1}} = \overline{\tau_{2}}\, \overline{\tau_{1}}$, as desired.   
\end{proof}

\begin{remark}\label{RemarkManin}
At this point we should recall that there is another construction in the literature related to the universal Hopf algebra of a comodule algebra, namely Manin's universal coacting Hopf algebra (see for instance \cite[Proposition 1.3.8, Remark 2.6.4]{Pareigis}). However, the latter construction is different from the one introduced in the present paper as, roughly speaking, it involves all possible coactions on a certain algebra not only those equivalent to a given one. More precisely, if $A$ is a given algebra, it can be easily seen that the universal coacting Hopf algebra of $A$ is precisely the initial object in the category whose objects are all comodule algebra structures on $A$ and the morphisms between two such objects $(H_{1},\,\rho_{1})$ and $(H_{2},\,\rho_{2})$ are Hopf algebra homomorphisms $f: H_{1} \to H_{2}$ such that the following diagram commutes:
\begin{equation*}
\xymatrix{ A  \ar[r]^(0.4){\rho_{1}} \ar[rd]_{\rho_{2}}& A \otimes H_{1} \ar[d]^{\id_A \otimes f} \\
& A\otimes H_{2}}
\end{equation*}
In other words, our construction can be considered as a refinement of Manin's as we get more information
on each class of equivalent coactions. Furthermore, note that Manin's universal coacting Hopf algebra was shown to exist only for finite dimensional algebras while our construction can be performed for any arbitrary algebra. 
\end{remark}

\subsection{$H$-comodule Hopf--Galois extensions}
\label{SectionHcomodHopfGalois}

Let $H$ be a Hopf algebra, let $A$ be a nonzero unital $H$-comodule algebra and let $\rho \colon A \to A\otimes H$ be its comodule map. 
Denote by $A^{\text{co}H}$ the subalgebra of coinvariants, i.e. 
$A^{\text{co}H} := \lbrace a\in A \mid \rho(a)=a\otimes 1_H \rbrace$.
Recall that $A$ is called a \textit{Hopf--Galois extension} of $A^{\text{co}H}$
if the linear map $\mathsf{can} \colon A \mathbin{\otimes_{A^{\text{co}H}}} A
\to A \mathbin{\otimes} H$ defined below is bijective: $$\mathsf{can}(a\otimes b) := 
ab_{(0)}\otimes b_{(1)}.$$
Our next result computes the universal Hopf algebra of a Hopf-Galois extension. 

\begin{theorem}\label{TheoremHComodHopfGaloisUnivHopfAlg}
Let $A/A^{\text{co}H}$ be a Hopf--Galois extension. Then $(H,\rho)$ is the universal Hopf algebra of $\rho$.
\end{theorem}
\begin{proof}
Note that the surjectivity of $\mathsf{can}$ implies $C(\rho)=H$. Hence
for every Hopf algebra $H_1$ and every $H_1$-comodule structure $\rho_1$ on $A$
equivalent to $\rho$, there exists a unique coalgebra homomorphism $\tau \colon H \to H_1$
such that
\begin{equation*}\xymatrix{ A  \ar[r]^(0.4){\rho} \ar[rd]_{\rho_1}& A \otimes H \ar[d]^{\id_A \otimes \tau} \\
& A\otimes H_1}
\end{equation*}

The only thing left to prove now is that $\tau$ is a Hopf algebra homomorphism.

First, since by Theorem~\ref{TheoremHopfUnivComodUnitality} unital comodule structures
can be equivalent only to unital comodule structures, $\rho_1(1_A)=1_A \otimes 1_{H_1}$.
At the same time we have $\rho_1(1_A)=(\id_A \otimes \tau)\rho(1_A)=1_A \otimes \tau(1_H)$.
Hence $\tau(1_H)=1_{H_1}$.

Second, for every $a,b\in A$ we have \begin{equation}\label{EqHopfGaloisPart0}a_{(0)}b_{(0)}\otimes \tau(a_{(1)}b_{(1)}) = \rho_1(ab)=
\rho_1(a)\rho_1(b)=a_{(0)}b_{(0)}\otimes \tau(a_{(1)})\tau(b_{(1)}).\end{equation}
We claim that
\begin{equation}\label{EqHopfGaloisPart1} a b_{(0)}\otimes \tau(h b_{(1)}) = a b_{(0)}\otimes \tau(h) \tau(b_{(1)})
\text{ for }a, b\in A,\ h\in H.\end{equation}
Choose a basis $(h_\alpha)_\alpha$ in $H$ and fix some basis element $h_\beta$.
The surjectivity of $\mathsf{can}$ implies that there exist
$a_i, b_i \in A$ such that $a \otimes h_\beta = \sum_i a_i b_{i(0)} \otimes b_{i(1)}$.
Note that
$\rho(b_i)= \sum_\alpha b_{i\alpha} \otimes h_\alpha$ for some $b_{i\alpha} \in A$
where for each $i$ only a finite number of $b_{i\alpha}$ is non-zero.
Hence $a \otimes h_\beta = \sum_{i,\alpha} a_i b_{i\alpha} \otimes h_\alpha$
and \begin{equation}\label{EqHopfGaloisSokr}\sum_i a_i b_{i\alpha} = \left\lbrace\begin{array}{rrr} a & \text{ if } &  \alpha = \beta, \\ 
0 & \text{ if } &  \alpha \ne \beta.  \end{array}\right.\end{equation}
Thus
\begin{equation*}\begin{split}  a b_{(0)}\otimes \tau(h_\beta) \tau(b_{(1)})
= \sum_{i,\alpha} a_i b_{i\alpha} b_{(0)} \otimes \tau(h_\alpha) \tau(b_{(1)})
 \\ \stackrel{(\ref{EqHopfGaloisPart0})}{=}
\sum_{i,\alpha} a_i b_{i\alpha} b_{(0)} \otimes \tau(h_\alpha b_{(1)})
\stackrel{(\ref{EqHopfGaloisSokr})}{=} 
a b_{(0)} \otimes \tau(h_\beta b_{(1)}).
\end{split}\end{equation*}
In other words, we have proved~(\ref{EqHopfGaloisPart1}) in the case $h=h_\beta$.
Since $\beta$ was an arbitrary index and both sides of~(\ref{EqHopfGaloisPart1})
are linear in $h$, (\ref{EqHopfGaloisPart1}) is proved for arbitrary $h$.

Fix now arbitrary $h\in H$ and some basis element $h_\gamma$. Again the surjectivity of $\mathsf{can}$ implies that there exists
$c_i, d_i \in A$ such that $1_A \otimes h_\gamma = \sum_i c_i d_{i(0)} \otimes d_{i(1)}$.
We can rewrite $\rho(d_i)=\sum_{\alpha} d_{i\alpha} \otimes h_\alpha$
for some $d_{i\alpha} \in A$. Then $1_A \otimes h_\gamma = \sum_{i,\alpha} c_i d_{i\alpha} \otimes h_\alpha$
and \begin{equation}\label{EqHopfGaloisSokr2}\sum_i c_i d_{i\alpha} = \left\lbrace\begin{array}{rrr} 1_A & \text{ if } &  \alpha = \gamma, \\ 
0 & \text{ if } &  \alpha \ne \gamma.  \end{array}\right.\end{equation}
Then \begin{equation*}\begin{split} 1_A \otimes \tau(h)\tau(h_\gamma)= \sum_{i,\alpha} c_i d_{i\alpha} \otimes \tau(h)\tau(h_\alpha)
\stackrel{(\ref{EqHopfGaloisPart1})}{=} \sum_{i,\alpha} c_i d_{i\alpha} \otimes \tau(h h_\alpha)
\stackrel{(\ref{EqHopfGaloisSokr2})}{=} 1_A \otimes \tau(h h_\gamma).
\end{split}\end{equation*}
Hence $\tau(h h_\gamma)=\tau(h)\tau(h_\gamma)$. Since $\gamma$ was an arbitrary index, $\tau$ is a bialgebra homomorphism and therefore a Hopf algebra homomorphism. Thus $(H,\rho)$ is the universal Hopf algebra of~$\rho$.
\end{proof}

\begin{remark}
In the proof of Theorem~\ref{TheoremHComodHopfGaloisUnivHopfAlg} we have used only the surjectivity of $\mathsf{can}$.
\end{remark}

If we consider the standard $G$-grading on the group algebra $FG$ of a group $G$,
then the universal group of this grading is isomorphic to $G$. In the case of a comodule structure
we can obtain a similar result:

\begin{corollary}\label{CorollaryUniversalHopfOfCoactionOnItselfIsItself}
Let $H$ be a Hopf algebra. Then the universal Hopf algebra
of the $H$-comodule algebra structure on $H$ defined by the comultiplication $\Delta \colon H \to H \otimes 
H$ is again  $(H, \Delta)$.
\end{corollary}
\begin{proof}
We have $H^{\text{co}H}  \cong F$ and $H/F$ is a Hopf--Galois extension by \cite[Examples 6.4.8, 1)]{Danara}.
The desired conclusion now follows from Theorem~\ref{TheoremHComodHopfGaloisUnivHopfAlg}.
\end{proof}

\begin{example} Let $H$ be a Hopf algebra and let $A$ be a unital $H$-module algebra. We denote by $A \# H$ the corresponding \textit{smash product}, i.e. $A \# H = A \otimes H$ as a vector space with multiplication given as follows:
$$
(a \# h)(b \# g) = a(h_{(1)} b) \# h_{(2)} g
$$
where we denote the element $a \otimes h \in A \otimes H$ by $a \# h$. Then, we have an $H$-comodule algebra structure on $A \# H$ given by:
$$
\rho\colon A \# H \to (A \# H) \otimes H,\,\,\, \rho(a \# h) = a \# h_{(1)} \otimes h_{(2)}
$$
Then $(A \# H)^{\mathrm{co}H} \cong A$ and $A \# H/A$ is a Hopf--Galois extension by \cite[Examples 6.4.8, 2)]{Danara}.
Now Theorem~\ref{TheoremHComodHopfGaloisUnivHopfAlg} implies that the universal Hopf algebra of $\rho$ is $(H,\, \id_{A} \otimes \Delta)$. 
\end{example}

\section{Module structures on algebras}

\subsection{Support equivalence of module structures on algebras}
\label{SectionHmodule}

Analogously, we can introduce the notion of equivalence of module structures on algebras.

\begin{definition}
Let $A_i$ be $H_i$-module algebras for Hopf algebras $H_i$, $i=1,2$. We say that an isomorphism
$\varphi \colon A_1 \mathrel{\widetilde\to} A_2$ of algebras is a \textit{support equivalence
of module algebra structures} on $A_1$ and $A_2$ if 
\begin{equation}\label{EqImagesOfHiCoincide}
\tilde\varphi\Bigl(\zeta_1\bigl(H_1 \bigr)\Bigr)=\zeta_2\bigl(H_2\bigr)
\end{equation}
where $\zeta_i$ is the module algebra structure on $A_{i}$ and the isomorphism $\tilde\varphi \colon \End_{F}(A_1) \mathrel{\widetilde\to} \End_{F}(A_2)$ is defined by the conjugation by $\varphi$.
In this case we call module algebra structures on $A_1$ and $A_2$
\textit{support equivalent via the isomorphism $\varphi$} and, as in the comodule algebra case, we will say  just \textit{equivalent} for short.\end{definition}

It is easy to see that each equivalence of module algebra structures maps $H_1$-submodules to $H_2$-submodules.
In Lemma~\ref{LemmaHEquivCodimTheSame} below we show that one can identify the corresponding relatively free $H_1$- and $H_2$-module algebras and codimensions of polynomial $H$-identities for equivalent algebras coincide.

As in the previous sections, we can restrict our consideration to the case when $A_1=A_2$
and $\varphi$ is an identity map. 

Let $A$ be a left $H$-module algebra and $\zeta \colon H \to  \End_{F}(A)$ the corresponding module map.
Denote by $\tilde x$ the image of $x\in H$ in $H/\ker \zeta$.

 Then the map $\hat \zeta \colon H/{\rm ker}\, \zeta \to \End_{F}(A)$ 
where 
$$
\hat \zeta(\tilde x) := \zeta(x)
$$  
for all $x\in H$,
defines on $A$ a structure of an $H/\ker \zeta$-module. Notice that $\hat \zeta$ is obviously injective and, moreover, we have $\zeta(H) = \hat \zeta(H/{\rm ker}\, \zeta)$.
Proposition~\ref{PropositionHmodEquivCriterion} is an analog of Proposition~\ref{PropositionHcomodEquivCriterion} for module algebras.

\begin{proposition}\label{PropositionHmodEquivCriterion}
Let $A_i$ be $H_i$-module algebras for Hopf algebras $H_i$, $i=1,2$. Then an isomorphism
$\varphi \colon A_1 \mathrel{\widetilde\to} A_2$ of algebras is an equivalence of module algebra structures $\zeta_i \colon H_i \to \End_{F}(A_i)$, $i=1,2$, if and only if there exists an isomorphism $\lambda: H_{1}/{\rm ker}\, \zeta_{1} \mathrel{\widetilde\to} H_{2}/{\rm ker}\, \zeta_{2}$ of algebras such that the following diagram is commutative:
\begin{equation}%\label{EqC(rho)Equiv}
\xymatrix{ H_{1}/{\rm ker}\, \zeta_{1} \ar[d]_{\lambda} \ar[r]^-{\hat \zeta_{1}} & \End_{F}(A_1) \ar@<-1ex>[d]^{\tilde\varphi} \\
H_{2}/{\rm ker}\, \zeta_{2} \ar[r]_-{\hat \zeta_{2}} & \End_{F}(A_2)}
\end{equation}
\end{proposition}

\begin{proof}
Suppose there exists $\lambda: H_{1}/{\rm ker}\, \zeta_{1} \mathrel{\widetilde\to} H_{2}/{\rm ker}\, \zeta_{2}$ an isomorphism of algebras such that diagram~(\ref{PropositionHmodEquivCriterion}) is commutative. Hence, as $\lambda$ is in particular surjective, we obtain
$$
\tilde\varphi\Bigl(\zeta_1\bigl(H_1 \bigr)\Bigr) = \tilde\varphi\Bigl(\hat \zeta_1\bigl(H_{1}/{\rm ker}\, \zeta_{1} \bigr)\Bigr) = \hat \zeta_{2}\Bigl(\lambda\bigl(H_{1}/{\rm ker}\, \zeta_{1} \bigr)\Bigr) = \hat \zeta_{2}\bigl(H_{2}/{\rm ker}\, \zeta_{2} \bigr) = \zeta_{2}(H_{2}) 
$$
as desired. Therefore, the module algebra structures on $A_1$ and $A_2$
are equivalent via the isomorphism $\varphi$. 

Conversely, assume now that the isomorphism of algebras $\varphi \colon A_1 \mathrel{\widetilde\to} A_2$ is an equivalence of module algebra structures on $A_1$ and $A_2$. 
Then if we identify $H_1/\ker \zeta_1$ with $\zeta_1(H_1)$ and $H_2/\ker \zeta_2$ with $\zeta_2(H_2)$,
we can take $\lambda$ to be the restriction of $\tilde\varphi$ on $\zeta_1(H_1)$.
In other words, $\lambda(\tilde x) = \tilde y_{x}$, where $y_{x} \in H_{2}$ such that $\tilde\varphi\bigl(\hat \zeta_{1}(\tilde x)\bigr) = \hat\zeta_{2}(\tilde y_{x})$. Then $\lambda$ is a well-defined algebra isomorphism which makes diagram~(\ref{PropositionHmodEquivCriterion}) commutative. 
\end{proof}

Note that if an $H$-module algebra $A$ is unital then the identity element $1_A$ is a common eigenvector for all operators from $H$
and it retains this property for all equivalent module algebra structures. The proposition below
implies that if the original module algebra structure is unital, then all module algebra structures
equivalent to it are unital too.

\begin{proposition}\label{PropositionHopfUnivModUnitality} Let $A$ be an $H$-module algebra for some Hopf algebra $H$. Suppose there exists
the identity element $1_A \in A$ such that $1_A$ is a common eigenvector for all operators from $H$.
Then $A$ is a unital $H$-module algebra.
\end{proposition}
\begin{proof}
Denote by $\lambda \in H^*$ the linear function such that $h1_A = \lambda(h)1_A$. Since $A$ is an $H$-module algebra, we have $\lambda(h_1 h_2)=\lambda(h_1)\lambda(h_2)$ and $\lambda(h)=\lambda(h_{(1)})\lambda(h_{(2)})$
for all $h,h_1,h_2 \in H$. Moreover, $\lambda(1_H)=1_F$.
 Hence \begin{equation*}\begin{split}
\lambda(h) =\lambda(h_{(1)})\varepsilon(h_{(2)})\lambda(1_H)=
  \lambda(h_{(1)})\lambda(h_{(2)})\lambda(Sh_{(3)}) \\=
 \lambda(h_{(1)})\lambda(Sh_{(2)})  =\lambda(h_{(1)}(Sh_{(2)}))
 =\varepsilon(h)\lambda(1_H)=\varepsilon(h)\end{split}\end{equation*}
 and $A$ is a unital $H$-module algebra.
\end{proof}

It is well known that if $H$ is a finite dimensional Hopf algebra, then
$H^*$ is a Hopf algebra too and the notions of an $H$-module and $H^*$-comodule
algebras coincide (\cite[Proposition 6.2.4]{Danara}).

\begin{proposition}
Let $\rho_1 \colon A \to A \otimes H_1$ and $\rho_2 \colon A \to A \otimes H_2$
be two comodule structures on an algebra $A$
where $H_1$ and $H_2$ are finite dimensional Hopf algebras.
Let $\zeta_i \colon H_i^* \to \End_F(A)$, $i=1,2$, be the corresponding homomorphisms of algebras. Then $\rho_1$ and $\rho_2$ are equivalent
comodule structures if and only if $\zeta_1$ and $\zeta_2$ are equivalent
module structures.
\end{proposition}
\begin{proof}
Follows directly from the definitions.
\end{proof}

\subsection{Universal Hopf algebra of a module algebra structure}\label{univmodalg}

Analogously to the case of comodule algebras, if $A$ is an $H$-module algebra for a Hopf algebra $H$ and $\zeta \colon H \to \End_F(A)$
is the corresponding algebra homomorphism,
one can consider the category ${}_{H}\mathcal C_A$ where 
\begin{enumerate}
\item the objects are
$H_1$-module algebra structures on the algebra $A$ for arbitrary
Hopf algebras $H_1$ over $F$ such that $\zeta_1(H_1)=\zeta(H)$
where $\zeta_1 \colon H_1 \to \End_F(A)$
is the algebra homomorphism corresponding to the $H_1$-module algebra structure on $A$;
\item the morphisms from an $H_1$-module algebra structure on $A$
with the corresponding homomorphism $\zeta_1$
to an $H_2$-module algebra structure with the corresponding homomorphism $\zeta_2$
are all Hopf algebra homomorphisms $\tau \colon H_1 \to H_2$
such that the following diagram is commutative:
$$\xymatrix{ \End_{F}(A)   & H_1 \ar[l]_(0.3){\zeta_1} \ar[d]^\tau \\
& \ar[lu]^{\zeta_2}H_2
}
$$
\end{enumerate}

Recall that by $R \colon  \mathbf{Alg}_F \to \mathbf{Hopf}_F$ 
we denote the right adjoint functor for the forgetful functor $U\colon \mathbf{Hopf}_F \to \mathbf{Alg}_F$. The counit of this adjunction is denoted by $\mu\colon UR \Rightarrow \id_{\mathbf{Alg}_F} $.

Suppose $\zeta_1 \colon H_1 \to \End_F(A)$ is a structure of an $H_1$-module algebra on $A$ that is equivalent to $\zeta$. Then $\zeta_1(H_1)=\zeta(H)$
and there exists a unique Hopf algebra homomorphism
$\varphi_1 \colon H_1 \to R(\zeta(H))$ such that the following diagram commutes:
\begin{equation*}
\xymatrix{ H_1 \ar[rd]_{\zeta_1} \ar[r]^-{\varphi_1} & R(\zeta(H)) \ar[d]^{\tilde\mu} \\
 & \zeta(H)}
\end{equation*}
where we denote $\tilde\mu = \mu_{\zeta(H)}$.

Now consider a subalgebra $H_\zeta$ of $R(\zeta(H))$ generated by $\varphi_1(H_1)$ for all such structures
$\zeta_1$. Obviously, $H_\zeta$ is a Hopf algebra. Let $\psi_\zeta := \tilde\mu\bigl|_{H_\zeta}$.

\begin{theorem}\label{TheoremFinExistenceHMod} The homomorphism $\psi_\zeta$ defines on $A$ a structure
of an $H_\zeta$-module algebra which is the terminal object
in ${}_{H}\mathcal C_A$.
\end{theorem}
\begin{proof}
An arbitrary element $h_0$ of $H_\zeta$ can be presented as a linear combination
of elements $\varphi_1(h_1)\cdots\varphi_n(h_n)$ where $n\in\mathbb N$, $h_i\in H_i$,
$H_i$ are Hopf algebras, $\varphi_i \colon H_i \to R(\zeta(H))$
 are homomorphisms of Hopf algebras, and each homomorphism $\tilde\mu \varphi_i$  defines on $A$ a structure of an $H_i$-module algebra equivalent to $\zeta$.
 Hence \begin{equation*}\begin{split}h_0(ab)=\varphi_1(h_1)\cdots\varphi_n(h_n)(ab)=h_1(h_2(\ldots h_n(ab)\ldots))
 \\=(h_{1(1)}h_{2(1)}\ldots h_{n(1)} a)
(h_{1(2)}h_{2(2)}\ldots h_{n(2)} b) \\= 
(\varphi_1(h_{1(1)})\varphi_2(h_{2(1)})\ldots \varphi_n(h_{n(1)}) a)
(\varphi_1(h_{1(2)})\varphi_2(h_{2(2)})\ldots \varphi_n(h_{n(2)}) b)
\\=\left(\left(\varphi_1(h_1)\cdots\varphi_n(h_n)\right)_{(1)}a\right)
\left(\left(\varphi_1(h_1)\cdots\varphi_n(h_n)\right)_{(2)}b\right)
\text{ for all }a,b\in A.\end{split}\end{equation*}
Note that the comultiplication in the middle is calculated each time in the corresponding Hopf algebra $H_i$.
Also we have used that each $\varphi_i$ is, in particular, a homomorphism of coalgebras.
Therefore $A$ is an $H_\zeta$-module algebra.

Now the choice of $H_\zeta \subseteq R(\zeta(H))$ implies that $(H_\zeta, \psi_\zeta)$ is the terminal object of ${}_H\mathcal C_A$.
\end{proof}

We call $(H_\zeta, \psi_\zeta)$ the \textit{universal Hopf algebra} of $\zeta$.

If $\zeta_1$ and $\zeta_2$ are two module structures on $A$
and $\zeta_2(H_2) \subseteq \zeta_1(H_1)$, we say that $\zeta_1$ is \textit{finer} than $\zeta_2$
and $\zeta_2$ is \textit{coarser} than $\zeta_1$.
Again, $\id_A$ is an equivalence of $\zeta_1$ and $\zeta_2$
if and only if $\zeta_1$ is both finer and coarser than $\zeta_2$.
 Note that 
we could define $H_\zeta$ as a subalgebra of $R(\zeta(H))$
generated by the images of all Hopf algebras whose action is coarser than $\zeta$.
Then the action of $H_\zeta$ would still be equivalent to $\zeta$, but
the proof of Theorem~\ref{TheoremFinExistenceHMod} would imply that
$H_\zeta$ is universal not only among module structures equivalent to $\zeta$, but also
among all the module structures coarser than $\zeta$. The uniqueness of $H_\zeta$
implies that the original $H_\zeta$ satisfies this property too, i.e. 
$H_\zeta$ is universal among all the module structures coarser than $\zeta$.

As in the case of comodule algebra structures, we will see that taking the universal Hopf algebra of a module algebra yields a functor. Indeed, given an algebra $A$ we define the category ${}_{A}\mathcal C$ as follows:
\begin{enumerate}
\item the objects are pairs $(H,\, \zeta)$ where $H$ is a Hopf algebra and $\zeta \colon H \to \End_{F}(A)$ is a left $H$-module algebra structure on $A$;

\item the morphisms between two objects $(H,\, \zeta)$ and $(H',\, \zeta')$ are algebra homomorphisms $\lambda
\colon \zeta(H) \to \zeta'(H')$ such that the following diagram is commutative:
\begin{equation*}
\xymatrix{ \zeta(H)  \ar[d]_{\lambda}  \ar[r]^{\hat\zeta} & \End_{F}(A)\\
\zeta'(H') \ar[ur]_-{\hat\zeta'}  & {}}
\end{equation*}
\end{enumerate} 
(Here $\hat\zeta$ and $\hat\zeta'$ are natural embeddings.)

Note that the existence of an arrow from  $(H,\, \zeta)$ to $(H',\, \zeta')$
just means that $\zeta$ is coarser than $\zeta'$ and ${}_{A}\mathcal C$ is just the preorder
of all module structures on $A$ with respect to the relation ``coarser/finer''.

Moreover, any morphism $\lambda
\colon \zeta(H) \to \zeta'(H')$ in ${}_{A}\mathcal C$ induces a unique Hopf algebra homomorphism $\overline{\lambda} \colon R(\zeta(H)) \to R(\zeta'(H'))$ such that the following diagram commutes:
\begin{equation*}
\xymatrix{ R(\zeta(H))  \ar[d]_{\overline{\lambda}}  \ar[rr]^{\mu_{\zeta(H)}} & {}& {\zeta(H)} \ar[d]_\lambda\\
R(\zeta'(H'))\ar[rr]_-{\mu_{\zeta'(H')}}  & {} &{\zeta'(H')}}
\end{equation*}

We claim that $\overline{\lambda}_{\big|H_\zeta} \subseteq H'_{\zeta'}$. 
Indeed, the existence of the arrow $\lambda$ implies that $\zeta'$ is coarser than $\zeta$.
Hence, if an $H_1$-module structure $\zeta_1$ is coarser than $\zeta$,
then $\zeta_1$ is coarser than $\zeta'$ too and $\bar\lambda$ maps the image of $H_1$
in $R(\zeta(H))$ to $H'_{\zeta'}$.

We can now define the desired functor:
\begin{theorem}\label{functorial2}
There exists a functor $G \colon {}_{A}\mathcal C \to  \mathbf{Hopf}_F$ given as follows:
$$
G(H,\, \zeta) = H_\zeta \,\, {\rm and}\,\, G(\lambda) = \overline{\lambda}_{\big|H_\zeta}.
$$
\end{theorem}
\begin{proof}
We only need to prove that $G$ respects composition of morphisms. To this end, let $(H,\, \zeta)$, $(H',\, \zeta')$ and $(H'',\, \zeta'')$ be objects in ${}_{A}\mathcal C$ and $\lambda_{1}: \zeta(H) \to
\zeta'(H')$, $\lambda_{2}: \zeta'(H') \to \zeta''(H'')$  two morphisms in ${}_{A}\mathcal C$. Then $G(\lambda_{1}) = \overline{\lambda_{1}}_{\big|H_\zeta}$ and $G(\lambda_{2}) = \overline{\lambda_{2}}_{\big|H'_{\zeta'}}$ where $ \overline{\lambda_{1}}$ and $ \overline{\lambda}_{2}$ are the unique Hopf algebra homomorphisms such that the following diagrams commute: 
\begin{equation}\label{functor4}
\xymatrix{ R(\zeta(H))  \ar[d]_{\overline{\lambda_{1}}}  \ar[rr]^{\mu_{\zeta(H)}} & {} & \zeta(H) \ar[d]^{\lambda_1}\\
R(\zeta'(H'))\ar[rr]_-{\mu_{\zeta'(H')}}  & {} &{\zeta'(H')}}\qquad \xymatrix{ R(\zeta'(H'))  \ar[d]_{\overline{\lambda_{2}}}  \ar[rr]^{\mu_{\zeta'(H')}} & {} & \zeta'(H') \ar[d]^{\lambda_2}\\
R(\zeta''(H''))\ar[rr]_-{\mu_{\zeta''(H'')}}  & {} &{\zeta''(H'')}}
\end{equation}
Moreover,  $G(\lambda_{2} \lambda_{1}) = \overline{\lambda_{2} \lambda_{1}}_{\big|H_\zeta}$ where $\overline{\lambda_{2} \lambda_{1}}$ is the unique Hopf algebra homomorphism which makes the following diagram commutative:
$$
\xymatrix{ R(\zeta(H))  \ar[d]_{\overline{\lambda_{2} \lambda_{1}}}  \ar[rr]^{\mu_{\zeta(H)}} & {} & \zeta(H)
\ar[d]^{\lambda_2 \lambda_1}\\
R(\zeta''(H''))\ar[rr]_-{\mu_{\zeta''(H'')}}  & {} &{\zeta''(H'')}}\
$$
Using~(\ref{functor4}) one can easily check that $\overline{\lambda_{2}} \, \overline{\lambda_{1}}$ makes the above diagram commutative as well and therefore we obtain $G(\lambda_{2} \lambda_{1}) = G(\lambda_{2})G(\lambda_{1})$. This finishes the proof.
\end{proof}

\begin{remark}\label{RemarkManin2}
Inspired by Manin's universal coacting Hopf algebra of an algebra, we can construct the universal acting Hopf algebra of an algebra. More precisely, any given algebra $A$ can be endowed with a $M(A,\,A)$-module algebra structure $\theta$, where $M(A, A)$ is the universal measuring bialgebra of $A$ (see \cite[Chapter VII]{Sw}). By \cite[Theorem 3.1]{CH}, there exists a Hopf algebra $H_{*}(M(A, A))$ together with a bialgebra homomorphism $\beta \colon H_{*}(M(A, A)) \to M(A, A)$ such that for any other bialgebra homomorphism $f\colon H \to M(A, A)$ from a Hopf algebra $H$ to $M(A, A)$ there exists a unique Hopf algebra homomorphism $g\colon H \to H_{*}(M(A, A))$ which makes the following diagram commutative:
\begin{equation*}
\xymatrix{ {H_{*}(M(A, A))}\ar[r]^-(0.4){\beta} & M(A, A) \\
H \ar[u]^{g} \ar[ur]_-{f} & {}}
\end{equation*}
Thus $(H_{*}(M(A,\,A)),\, \theta\,(\beta \otimes 1_{A}))$ is the terminal object in the category whose objects are all module algebra structures on $A$ and the morphisms between two such objects $(H_{1},\, \psi_{1})$ and $(H_{2},\, \psi_{2})$ are Hopf algebra homomorphisms $f \colon H_{1} \to H_{2}$ such that the following diagram is commutative: 
\begin{equation*}
\xymatrix{ {H_{1}\otimes A}\ar[r]^-(0.4){\psi_{1}}\ar[d]_{f \otimes 1_{A}} & A \\
{H_{2} \otimes A}\ar[ur]_-{\psi_{2}} & {}}
\end{equation*}
We call $(H_{*}(M(A,\,A)),\, \theta\,(\beta \otimes 1_{A}))$ the \textit{universal acting Hopf algebra of the algebra $A$}. For further details as well as the construction of the (co)universal acting Hopf algebra of a coalgebra we refer to \cite{AGV}.
\end{remark}

Theorem~\ref{TheoremUniversalHopfOfActionOnHStarIsItself} below
provides an analog of Corollary~\ref{CorollaryUniversalHopfOfCoactionOnItselfIsItself} for module structures. 

\begin{theorem}\label{TheoremUniversalHopfOfActionOnHStarIsItself}
Let $H$ be a Hopf algebra. Denote by $\zeta \colon H \to \End_F(H^*)$ the homomorphism
defined by $(\zeta(h)\lambda)(t):=\lambda(th)$ for all $h,t\in H$, $\lambda\in H^*$.
Then  $\zeta$ is a unital $H$-module structure on the algebra $H^*$ and
the universal Hopf algebra of $\zeta$ is again $(H, \zeta)$.
\end{theorem}
\begin{proof}
Indeed, if $\lambda,\mu \in H^*$, we have \begin{equation*}\begin{split}(\zeta(h)(\lambda\mu))(t)
=(\lambda\mu)(th) =\lambda(t_{(1)} h_{(1)}) \mu(t_{(2)} h_{(2)})
\\ = (h_{(1)}\lambda)(t_{(1)})(h_{(2)}\mu)(t_{(2)})=((h_{(1)}\lambda)(h_{(2)}\mu))(t)
\end{split}\end{equation*}
for all $h,t\in H$ and $h\varepsilon = \varepsilon(h)\varepsilon$. Hence 
$\zeta$ is a unital $H$-module structure on the algebra $H^*$.

Consider a $H_1$-module structure $\zeta_1 \colon H_1 \to \End_F(H^*)$ such
that $\zeta_1(H_1)=\zeta(H)$. Then Proposition~\ref{PropositionHopfUnivModUnitality}
implies that $\zeta_1$ is a unital module structure.
Note that $\zeta$ is an embedding.
Hence if we restrict the codomain $\End_{F}(H^*)$ of $\zeta$ and $\zeta_1$ to $\zeta_1(H_1)=\zeta(H)$,
the map $\zeta$ becomes an isomorphism between $H$ and $\zeta(H)$
and there exists exactly one unital homomorphism $\tau \colon H_1 \to H$ of algebras
such that the diagram below is commutative:
$$\xymatrix{ \End_{F}(H^*)   & H_1 \ar[l]_(0.3){\zeta_1} \ar[d]^\tau \\
& \ar[lu]^{\zeta}H
}
$$
This map $\tau$ satisfies the following equality:
\begin{equation}\label{EqTauTHEqualityHCoactsOnHStar}(\zeta_1(h)\lambda)(t)=\lambda(t\tau(h))
\end{equation}
for all $h\in H_1$, $t\in H$, $\lambda \in H^*$.
In order to show that $\tau$ is a homomorphism of Hopf algebras, it is enough
to show that $\tau$ is a homomorphism of coalgebras.
Substituting in~(\ref{EqTauTHEqualityHCoactsOnHStar}) $t=1_H$ and $\lambda=\varepsilon_H$, we get 
$\varepsilon_H(\tau(h))=(\zeta_1(h)\varepsilon_H)(1_H)=\varepsilon_{H_1}(h)\varepsilon_H(1_H)=\varepsilon_{H_1}(h)$.
Considering arbitrary $\lambda_1,\lambda_2 \in H^*$ and using~(\ref{EqTauTHEqualityHCoactsOnHStar})
once again,
we obtain \begin{equation*}\begin{split}\lambda_1(\tau(h_{(1)}))\lambda_2(\tau(h_{(2)}))=
(\zeta_1(h_{(1)})\lambda_1)(1_H)(\zeta_1(h_{(2)})\lambda_2)(1_H)
\\ = (\zeta_1(h_{(1)})\lambda_1 \otimes \zeta_1(h_{(2)})\lambda_2)(1_H \otimes 1_H)
= ((\zeta_1(h_{(1)})\lambda_1)(\zeta_1(h_{(2)})\lambda_2))(1_H)
\\ =(\zeta_1(h)(\lambda_1\lambda_2))(1_H)=(\lambda_1\lambda_2)(\tau(h))
=\lambda_1(\tau(h)_{(1)})\lambda_2(\tau(h)_{(2)}).\end{split}\end{equation*}
Since $\lambda_1,\lambda_2 \in H^*$ were arbitrary,
we get $\tau(h_{(1)}) \otimes \tau(h_{(2)}) = \tau(h)_{(1)}\otimes \tau(h)_{(2)}$
for all $h\in H$ and $\tau$ is indeed a homomorphism of coalgebras.
Therefore, $(H, \zeta)$ is the universal Hopf algebra of $\zeta$.
\end{proof}

\subsection{$H$-module Hopf--Galois extensions}
\label{SectionHmodHopfGalois}

In the sequel we prove a result analogous to Theorem~\ref{TheoremHComodHopfGaloisUnivHopfAlg} for $H$-module algebras.
Let $A$ be a nonzero unital $H$-module algebra for a Hopf algebra $H$ and consider $\zeta \colon H \to \End_F(A)$ to
be the corresponding $H$-action on $A$. Denote
by $A^H$ its subalgebra of invariants, i.e. 
$A^H := \lbrace a\in A \mid ha=\varepsilon(h)a \text{ for all }
h\in H \rbrace$. The algebra $A$ is a \textit{Hopf--Galois extension} of $A^H$
if the linear map $\mathsf{can} \colon A \mathbin{\otimes_{A^H}} A
\to \Hom_F(H, A)$ defined by: $$\mathsf{can}(a\otimes b)(h) := 
a(hb),$$ is injective and has a dense image in the \textit{finite topology}
(i.e. the compact-open topology on $\Hom_F(H, A)$ defined for the discrete topologies on $H$ and $A$).

\begin{theorem}\label{TheoremHModHopfGaloisUnivHopfAlg}
Let $A/A^H$ be a Hopf--Galois extension. Then $(H,\zeta)$ is the universal Hopf algebra of $\zeta$.
\end{theorem}
\begin{proof}
The density of the image of $\mathsf{can}$ implies that for every $h\ne 0$
there exist $a,b\in A$ such that $a(hb)\ne 0$. In other words, $\ker\zeta=0$.
Hence for every Hopf algebra $H_1$ and every $H_1$-module structure $\zeta_1$ on $A$ equivalent to $\zeta$ via $\id_A$, there exists a unique algebra homomorphism $\tau \colon H_1 \to H$
such that the diagram below is commutative:
\begin{equation*}\xymatrix{ H  \ar[r]^(0.35){\zeta} & 
\End_F(A)  \\
H_1 \ar[ru]_(0.4){\zeta_1} \ar[u]^{\tau} & }
\end{equation*}

The only thing we need to prove now is that $\tau$ is a Hopf algebra homomorphism.

First, since $A$ is a unital $H$-module algebra,
by Proposition~\ref{PropositionHopfUnivModUnitality},
$A$ is a unital $H_1$-module algebra too.
Hence
 $\varepsilon(h)
1_A = h1_A=\tau(h)1_A=\varepsilon(\tau(h))1_A$
for all $h\in H_1$. Therefore $\varepsilon(h)= \varepsilon(\tau(h))$.

Let $(h_\alpha)_\alpha$ be a basis in $H$.
We claim that 
for every $\lambda_{\alpha\beta}\in F$
such that only a finite number of them is nonzero,
the condition \begin{equation}\label{EqLambdaAlphaBetaHopfGaloisModUniv}\sum_{\alpha,\beta} \lambda_{\alpha\beta}(h_\alpha a)
(h_\beta b) = 0\text{ for all }a,b\in A\end{equation}
implies $\lambda_{\alpha\beta}=0$ for all $\alpha,\beta$.

Indeed, suppose~\eqref{EqLambdaAlphaBetaHopfGaloisModUniv}
holds.
Let $\Lambda$ be a finite set of indices such that $\lambda_{\alpha\beta}=0$
unless $\alpha,\beta \in \Lambda$.
 The density of the image of $\mathsf{can}$ implies that for every $\gamma$ there exist $a_{\gamma i},b_{\gamma i} \in A$
such that $$\sum_i a_{\gamma i}(h_\alpha b_{\gamma i})=\left\lbrace
\begin{array}{crl} 1_A & \text{ if } &  \alpha = \gamma,\\
0 & \text{ if } & \alpha \ne \gamma \text{ and } \alpha\in \Lambda.
\end{array} \right.$$
Then by~\eqref{EqLambdaAlphaBetaHopfGaloisModUniv}
for every $b\in A$ and every $\gamma$
we have $$ \sum_{\beta\in\Lambda} \lambda_{\gamma\beta} h_\beta b =
\sum_{\substack{\alpha,\beta\in \Lambda, \\ i}} \lambda_{\alpha\beta}a_{\gamma i}(h_\alpha b_{\gamma i})
(h_\beta b) = 0. $$
Now $\ker\zeta=0$ implies $\sum_\beta \lambda_{\gamma\beta} h_\beta = 0$
and $\lambda_{\gamma\beta}=0$ for all $\beta,\gamma$.

In virtue of~\eqref{EqModCompat}, for every $h\in H_1$ and $a,b \in A$
we have $$(\tau(h_{(1)})a)(\tau(h_{(2)})b)=(h_{(1)}a)(h_{(2)}b)=h(ab)=\tau(h)(ab)=
(\tau(h)_{(1)}a)(\tau(h)_{(2)}b).$$
By~\eqref{EqLambdaAlphaBetaHopfGaloisModUniv}
we obtain $\Delta(\tau(h))=(\tau\otimes \tau)\Delta(h)$
for all $h\in H_1$ and $\tau$ is a Hopf algebra homomorphism.
Thus $(H,\zeta)$ is indeed the universal Hopf algebra of $\zeta$.
\end{proof}
\begin{remark}
In the proof of Theorem~\ref{TheoremHModHopfGaloisUnivHopfAlg} we have used only the density of
the image of $\mathsf{can}$.
\end{remark}

\section{Applications}

\subsection{Rational actions of affine algebraic groups}
\label{SectionActionsAffAlgGroups}

Let $G$ be an affine algebraic group over an algebraically closed field $F$ and let $\mathcal O(G)$ be the algebra of regular functions on $G$. Then $\mathcal O(G)$ is a Hopf algebra where the comultiplication~$\Delta$ and the antipode~$S$ are induced by, respectively, the multiplication and taking the inverse element in $G$,
and the counit~$\varepsilon$ of~$\mathcal O(G)$ is just the calculation of the value at $1_G$. (See the details e.g. in~\cite[Chapter~4]{Abe}.) Suppose $G$ is acting \textit{rationally} by automorphisms on a finite dimensional algebra $A$, i.e. for a given basis $a_1, \ldots, a_n$ in $A$ there exist $\omega_{ij} \in \mathcal O(G)$, where $1\leqslant i,j \leqslant n$, such that $g a_j = \sum_{i=1}^n \omega_{ij}(g) a_i$
for all $1\leqslant j \leqslant n$ and $g\in G$.
This implies that $A$ is an $\mathcal O(G)$-comodule algebra where $\rho(a_j) := \sum_{i=1}^n 
 a_i \otimes \omega_{ij}$ for $1\leqslant j \leqslant n$.
  At the same time $A$ is an $\mathcal O(G)^\circ$-module algebra
  where $\mathcal O(G)^\circ$ is the finite dual of $\mathcal O(G)$:
$f^* a_j = \sum_{i=1}^n f^*(\omega_{ij}) a_i$ for all $1\leqslant j \leqslant n$ and $f^* \in \mathcal O(G)^\circ$.  
  The Lie algebra $\mathfrak g$ of $G$ is the subspace consisting of all \textit{primitive} elements of $\mathcal O(G)^\circ$, i.e. $f^* \in \mathcal O(G)^\circ$ such that $\Delta(f^*)=f^*\otimes 1 + 1\otimes f^*$, and the $\mathfrak g$-action on $A$ by derivations is just the restriction of the 
  $\mathcal O(G)^\circ$-action. At the same time the group $G$ itself
  can be identified with the group of \textit{group-like} elements of $\mathcal O(G)^\circ$, i.e.
  $f^* \in \mathcal O(G)^\circ$ such that $\Delta(f^*)=f^*\otimes f^*$ and $f^*\ne 0$.
  
  Hence three Hopf algebras are acting on $A$:
  $\mathcal O(G)^\circ$, $FG$ and $U(\mathfrak g)$ (the universal enveloping algebra of the Lie algebra~$\mathfrak g$). In Theorem~\ref{TheoremAffAlgGrAllEquiv} below
  we prove that all three actions are equivalent in the case when $G$ is connected. In order to show this, we need an auxiliary lemma, that is a generalization of \cite[Lemma~3]{GordienkoKochetov}.
  
%  \begin{lemma}[{\cite[Lemma~3]{GordienkoKochetov}}]\label{LemmaDensityFinDimImage} Let $V$ be a comodule over a coalgebra $C$. Let $\rho \colon V \to V \otimes C$ be the corresponding comodule map and let $\zeta \colon C^* \to \End_F(V)$ be the action of the algebra $C^*$
%  on $V$ defined by $c^* v := c^*(v_{(1)})v_{(0)}$ for $c^* \in C^*$ and $v\in V$.
%  Let $A \subseteq C^*$ be a \textbf{dense} subalgebra, i.e. $A^\perp := \lbrace c\in C \mid a(c)=0 
%  \text{ for all } a\in A\rbrace = 0$. Then $\zeta(A)=\zeta(C^*)$.
%  \end{lemma}
%  \begin{proof}
%  In light of \cite[Proposition A.5, c)]{Sw} we have:
%  \begin{eqnarray*}
%  A + C(\rho)^{\perp} = \bigl(A + C(\rho)^{\perp}\bigl)^{\perp \perp} = \bigl(A^{\perp} \cap C(\rho)^{\perp \perp}\bigl)^{\perp} = 0^{\perp} = C^{*}
%  \end{eqnarray*}
%  Since $C(\rho)^{\perp} = {\rm ker}(\zeta)$ we obtain $\zeta(A)=\zeta(C^*)$, as desired.
%  \end{proof}

  \begin{lemma}\label{LemmaDensityFinDimImage} Let $V$ be a comodule over a coalgebra $C$ with coaction map $\rho\colon V\to V\otimes C$. Suppose that there exists a finite dimensional subcoalgebra $D\subseteq C$ such that $\rho(V)\subseteq V\otimes D$ (e.g. $V$ is finite dimensional itself). Let $\zeta \colon C^* \to \End_F(V)$ be the corresponding action of the algebra $C^*$
  on $V$ defined by $c^* v := c^*(v_{(1)})v_{(0)}$ for $c^* \in C^*$ and $v\in V$.
  Let $A \subseteq C^*$ be a \textbf{dense} subalgebra, i.e. $A^\perp := \lbrace c\in C \mid a(c)=0 
  \text{ for all } a\in A\rbrace = 0$. Then $\zeta(A)=\zeta(C^*)$.
  \end{lemma}
  \begin{proof}
  It is sufficient to show that the restriction
  of both $C^*$ and $A$ to $D$ coincides with $D^*$. The first one is obvious since every linear function
  on $D$ can be extended to a linear function on the whole $C$. The second is proved as follows.
  The elements of $A$ viewed as linear functions on $D$ form a subspace $W$ in $D^*$. If $W \ne D^*$,
  the finite dimensionality of $D$ implies that there exists $d\in D$, $d\ne 0$, such that $w(d)=0$
  for all $w\in W$. As a consequence, $d \in A^\perp$ and we get a contradiction to the density
  of $A$ in $C^*$. Therefore, the restriction of $A$ to $D$ coincides with $D^*$, we have $\zeta(c^*)=\zeta(a)$ for any $a\in A$, $c^*\in C^*$
  such that $c^*\bigr|_{D} = a\bigr|_{D}$ which finally implies  $\zeta(A)=\zeta(C^*)$.
  \end{proof}
%  \begin{remark} The proof of Lemma~\ref{LemmaDensityFinDimImage} works for infinite dimensional $V$
%  with finite dimensional $D$ too.  
%  \end{remark}

Now we are ready to prove the theorem.

\begin{theorem}\label{TheoremAffAlgGrAllEquiv}
Let $G$ be a connected affine algebraic group over an algebraically closed field $F$ of characteristic $0$ acting rationally by automorphisms on a finite dimensional algebra $A$. Let $\mathfrak g$
 be the Lie algebra of $G$. Then the corresponding
 $FG$-action, $U(\mathfrak g)$-action and $\mathcal O(G)^\circ$-action on $A$ are equivalent.
\end{theorem}
\begin{proof}
By Lemma~\ref{LemmaDensityFinDimImage}, it is sufficient to prove that the images of $FG$,
$U(\mathfrak g)$ and $\mathcal O(G)^\circ$ are dense in $\mathcal O(G)^*$. For $FG$ this follows
from the definition of $\mathcal O(G)$. If we show that $U(\mathfrak g)$ is dense
in $\mathcal O(G)^*$, we will get automatically that $\mathcal O(G)^\circ$ is dense too, since $U(\mathfrak g) \subseteq \mathcal O(G)^\circ$.

Let $I := \ker \varepsilon$, the set of all polynomial functions from $\mathcal O(G)$ that take zero
at $1_G$. By~\cite[Proposition~9.2.5]{Montgomery}, $$U(\mathfrak g)=\lbrace a \in
\mathcal O(G)^\circ \mid a(I^n)=0 \text{ for some } n\in\mathbb N \rbrace=\mathcal O(G)'$$
where $\mathcal O(G)'$ is the irreducible component of $\varepsilon$ in $\mathcal O(G)^\circ$ (see~\cite[Definition~5.6.1]{Montgomery}). By~\cite[Chapter~II, Section~7.3]{HumphreysAlgGr} the connectedness of $G$ implies that $G$ is irreducible as a variety. Hence by a corollary
of Krull's theorem~\cite[Corollary~10.18]{AtiyahMacdonald} we have $\bigcap_{n\geqslant 1} I^n = 0$
and~\cite[Proposition~9.2.10]{Montgomery} implies that $U(\mathfrak g)=\mathcal O(G)'$
is dense in $\mathcal O(G)^*$. Hence $\mathcal O(G)^\circ$ is dense in $\mathcal O(G)^*$ too
and the $FG$-action, the $U(\mathfrak g)$-action and the $\mathcal O(G)^\circ$-action on $A$ are equivalent.\end{proof}

\subsection{Actions of cocommutative Hopf algebras}
\label{SectionActionsCocommHopfAlgebras}

If we restrict our consideration to actions of cocommutative Hopf algebras, we can define the
notion of a universal cocommutative Hopf algebra of a given action, i.e. such a Hopf algebra
whose action is universal among all actions of cocommutative Hopf algebras equivalent to a given one.
Besides its own interest, the universal cocommutative Hopf algebra can help calculating
the Hopf algebra $H$ that is universal among equivalent actions of all Hopf algebras, not necessarily
cocommutative ones, if one manages to prove that $H$ is cocommutative too, as we do in Proposition~\ref{PropositionDoubleNumbersUniversal} below.

Analogously to the case of arbitrary Hopf algebras, if $A$ is an $H$-module algebra for a 
cocommutative Hopf algebra $H$ and $\zeta \colon H \to \End_F(A)$
is the corresponding algebra homomorphism,
one can consider the category ${}_{H}\mathcal C^{\mathrm{coc}}_A$ where 
\begin{enumerate}
\item the objects are
$H_1$-module algebra structures on the algebra $A$ for cocommutative
Hopf algebras $H_1$ over $F$ such that $\zeta_1(H_1)=\zeta(H)$
where $\zeta_1 \colon H_1 \to \End_F(A)$
is the algebra homomorphism corresponding to the $H_1$-module algebra structure on $A$;
\item the morphisms from an $H_1$-module algebra structure on $A$
with the corresponding homomorphism $\zeta_1$
to an $H_2$-module algebra structure with the corresponding homomorphism $\zeta_2$
are all Hopf algebra homomorphisms $\tau \colon H_1 \to H_2$
such that the following diagram is commutative:
$$\xymatrix{ \End_{F}(A)   & H_1 \ar[l]_(0.3){\zeta_1} \ar[d]^\tau \\
& \ar[lu]^{\zeta_2}H_2
}
$$
\end{enumerate}

Suppose $\zeta_1 \colon H_1 \to \End_F(A)$ is a structure of an $H_1$-module algebra on $A$ that is equivalent to $\zeta$. Then $\zeta_1(H_1)=\zeta(H)$
and there exists a unique Hopf algebra homomorphism
$\varphi_1 \colon H_1 \to R(\zeta(H))$ such that the following diagram commutes:
\begin{equation*}
\xymatrix{ H_1 \ar[rd]_{\zeta_1} \ar[r]^-{\varphi_1} & R(\zeta(H)) \ar[d]^{\tilde\zeta} \\
 & \zeta(H)}
\end{equation*}
where we denote $\tilde\zeta = \mu_{\zeta(H)}$.

Now consider a subalgebra $H^{\mathrm{coc}}_\zeta$ of $R(\zeta(H))$ generated by $\varphi_1(H_1)$ for all such structures $\zeta_1$ where $H_1$ is a cocommutative Hopf algebra. Obviously, $H^{\mathrm{coc}}_\zeta$ is a Hopf algebra. Let $\psi^{\mathrm{coc}}_\zeta := \tilde\zeta\bigl|_{H^{\mathrm{coc}}_\zeta}$.

The same argument as in Theorem~\ref{TheoremFinExistenceHMod}
shows that $(H^{\mathrm{coc}}_\zeta, \psi^{\mathrm{coc}}_\zeta)$ is a terminal object in ${}_{H}\mathcal C^{\mathrm{coc}}_A$.
We call this terminal object the \textit{universal
cocommutative Hopf algebra} of $\zeta$. 

In the case of an algebraically closed field of characteristic $0$,
the Cartier--Gabriel--Kostant theorem (see e.g.~\cite[Corollary~5.6.4 and Theorem~5.6.5]{Montgomery})
allows us to give a concrete description of the universal cocommutative Hopf algebra, using a similar technique as in \cite{GranKadjoVercruysse}. 

\begin{theorem}\label{TheoremUniversalCocommutative} Let $A$ be an $H$-module algebra for a 
cocommutative Hopf algebra $H$ over an algebraically closed field $F$ of characteristic $0$
and let $\zeta \colon H \to \End_F(A)$ be the corresponding algebra homomorphism.
Let $$G_0 := \mathcal U(\zeta(H)) \cap \Aut(A)\text{ and }L_0 := \Der(A) \cap \zeta(H)$$
where $\mathcal U(\zeta(H))$ is the group of invertible elements of the algebra $\zeta(H)$
and $\Der(A)$ is the Lie algebra of derivations of $A$ viewed as a subspace of $\End_F(A)$.
Define $H_0$ to be the Hopf algebra which as a coalgebra is just $U(L_0) \otimes FG_0$,
as an algebra is the smash product $U(L_0) \mathbin{\#} FG_0$ (the $G_0$-action on $L_0$ is the standard one by conjugations) and the antipode is defined by 
$$S(w \mathbin{\#} g) := (1 \mathbin{\#} g^{-1})(Sw\mathbin{\#} 1)$$
for all $w \in U(L_0)$ and $g\in G_0$. Then $(H_0, \zeta_0)$ is the terminal object in 
${}_{H}\mathcal C^{\mathrm{coc}}_A$. (Here $\zeta_0$ is the $H_0$-action on $A$ induced by the $G$- and $L$-action on $A$.)
\end{theorem}
\begin{proof} By the Cartier--Gabriel--Kostant theorem, $H \cong U(L) \mathbin{\#} FG$ where $G$ is the group of group-like elements of $H$ and $L$ is the Lie algebra of primitive elements of $H$.
Since $\zeta(G) \subseteq G_0$, $\zeta(L) \subseteq L_0$ and $G$ and $L$ generate $H$ as an algebra, $G_0$ and $L_0$ generate $\zeta(H)$ and we have $\zeta_0(H_0)=\zeta(H)$, i.e. $\zeta_0$ is equivalent to $\zeta$.

Let $\zeta_1 \colon H_1 \to \End_F(A)$ be an action of another cocommutative Hopf algebra $H_1$ equivalent
to $\zeta$. Again, by the Cartier--Gabriel--Kostant theorem,  $H_1 \cong U(L_1) \mathbin{\#} FG_1$ where $G_1$ is the group of group-like elements of $H_1$ and $L_1$ is the Lie algebra of primitive
elements of $H_1$.
We have to show that there exists a unique Hopf algebra homomorphism $\tau \colon H_1 \to H_0$ such that $\zeta_0\tau =\zeta_1$.

It is not difficult to see (using, say, the techniques of~\cite[Proposition~5.5.3, 2)]{Montgomery}) that the Lie algebra of primitive elements of $H_0=U(L_0)\mathbin{\#} FG_0$ coincides with $L_0 \otimes 1$, which we identify with $L_0$,
while the group of group-like elements of $H_0$ coincides with $1 \otimes G_0$, which we identify with $G_0$.
Since $\tau$ (if it indeed exists) maps group-like elements to group-like ones and primitive elements to primitive ones,
we must have $\tau(G_1) \subseteq G_0$ and $\tau(L_1) \subseteq L_0$. Note that $\zeta_0\bigr|_{G_0} = \id_{G_0}$
and $\zeta_0\bigr|_{L_0} = \id_{L_0}$. Hence $\tau\bigr|_{G_1}$ and $\tau\bigr|_{L_1}$ are uniquely
determined by $\zeta_1\bigr|_{G_1} \colon G_1 \to G_0$ and $\zeta_1\bigr|_{L_1}  \colon L_1 \to L_0$. Since $H_1$ is generated as an algebra by $G_1$ and $L_1$,
there exists at most one Hopf algebra homomorphism $\tau \colon H_1 \to H_0$ such that $\zeta_0\tau =\zeta_1$.
Now we define $\tau$ to be the Hopf algebra homomorphism induced by $\zeta_1\bigr|_{G_1}$ and $\zeta_1\bigr|_{L_1}$.
\end{proof}

Below we provide a criterion for universal cocommutative and universal Hopf algebras to coincide:

\begin{theorem}\label{TheoremCocommUnivIsUniv} Let $A$ be an $H$-module algebra for a 
cocommutative Hopf algebra $H$ over an algebraically closed field $F$ of characteristic $0$
and let $\zeta \colon H \to \End_F(A)$ be the corresponding algebra homomorphism.
Then the universal
cocommutative Hopf algebra $(H^{\mathrm{coc}}_\zeta, \psi^{\mathrm{coc}}_\zeta)$ of $\zeta$ is its universal
Hopf algebra if and only if every Hopf algebra action $\zeta_1 \colon H_1 \to \End_F(A)$
equivalent to $\zeta$ factors through some cocommutative Hopf algebra $H_2$,
i.e. there exists a Hopf algebra action $\zeta_2 \colon H_2 \to \End_F(A)$
and a Hopf algebra homomorphism $\theta$ making the diagram below commutative:
\begin{equation*}\xymatrix{ H_1  \ar[r]^(0.4){\zeta_1} \ar@{-->}[d]^{\theta} & \End_F(A)  \\
H_2 \ar[ru]^{\zeta_2} &    
}
\end{equation*}
\end{theorem}
\begin{proof} The ``only if'' part follows immediately from the definition
of the universal Hopf algebra of an action.

Suppose that every Hopf algebra action
equivalent to $\zeta$ factors through some cocommutative Hopf algebra $H_2$.
Without loss of generality we may assume that the corresponding homomorphism $\theta$
is surjective and therefore the $H_2$-action is equivalent to $\zeta$.
Since every such $H_2$ is cocommutative, this implies that every Hopf algebra action
equivalent to $\zeta$ factors through $\psi^{\mathrm{coc}}_\zeta$.
Consider now the universal Hopf algebra $(H_\zeta, \psi_\zeta)$ of $\zeta$.
The universal property of
$\psi_\zeta$ and the fact that $\psi_\zeta$ factors through $\psi^{\mathrm{coc}}_\zeta$ too imply that there exist Hopf algebra homomorphisms $\theta_1, \theta_2$ making the diagram below
commutative:
\begin{equation*}\xymatrix{ H^{\mathrm{coc}}_\zeta  \ar[r]^(0.4){\psi^{\mathrm{coc}}_\zeta} \ar@<0.5ex>@{-->}[d]^{\theta_1} & \End_F(A)  \\
H_\zeta \ar[ru]^{\psi_\zeta} \ar@<0.5ex>@{-->}[u]^{\theta_2} &    
}
\end{equation*}

Now the uniqueness of the comparison maps in the definitions
of both $\psi_\zeta$ and $\psi^{\mathrm{coc}}_\zeta$
imply that $\theta_1 \theta_2 = \id_{H_\zeta}$
and $\theta_2 \theta_1 = \id_{H^{\mathrm{coc}}_\zeta}$.
In particular,
we may identify $(H_\zeta, \psi_\zeta)$  with $(H^{\mathrm{coc}}_\zeta, \psi^{\mathrm{coc}}_\zeta)$.
\end{proof}

Now we give an example of an action of a cocommutative Hopf algebra with a non-cocommutative
universal Hopf algebra. This will show that the universal cocommutative Hopf algebra
does not always coincide with the universal Hopf algebra.
In fact, it is sufficient to present an algebra $A$ with an $H$-action equivalent to
an action of a cocommutative Hopf algebra such that $(h_{(1)}a)(h_{(2)}b)\ne (h_{(2)}a)(h_{(1)}b)$
for some $h\in H$, $a,b \in A$. Then such an $H$-action can never factor through a cocommutative algebra. Let us present an example of this situation.

\begin{example}\label{ExampleCocommUnivDifferent} Let $F$ be a field, let $A=F1_A\oplus Fa\oplus Fb \oplus Fab$ be the $F$-algebra
where $a^2=b^2=ba=0$ and let $S_3$ be the $3$rd symmetric group. Consider the equivalent $S_3$- and the $\mathbb Z/4\mathbb Z$-gradings on $A$
defined by $$A^{(\mathrm{id})}:=A^{(\bar 0)}:= F 1_A,\
A^{\bigl((12) \bigr)}:=A^{(\bar 1)}:= F a,$$
$$A^{\bigl((23) \bigr)}:=A^{(\bar 2)}:= F b,\
A^{\bigl((123) \bigr)}:=A^{(\bar 3)}:= F ab.$$
Since the gradings are equivalent, by Theorem~\ref{TheoremGradEquivCriterion} the corresponding $(FS_3)^*$- and $(F(\mathbb Z/4\mathbb Z))^*$-actions are equivalent too,
while $(F(\mathbb Z/4\mathbb Z))^*$ is commutative cocommutative and $(FS_3)^*$ is 
commutative non-cocommutative. In view of the above, to prove that the universal Hopf algebra of these actions is not cocommutative, it suffices to show that there exists an element $h\in (FS_3)^*$ such that
$(h_{(1)}a)(h_{(2)}b) \ne (h_{(2)}a)(h_{(1)}b)$.
Recall that if $G$ is a finite group and $(h_g)_{g\in G}$ is the basis of $(FG)^*$ dual to the basis $(g)_{g\in G}$ of $FG$, the comultiplication on $(FG)^*$ is given by  
%then $(FG)^*$ admits a basis $(h_g)_{g\in G}$ dual to the basis $(g)_{g\in G}$. Moreover, 
$\Delta h_g = \sum\limits_{st=g} h_s \otimes h_t$.
Hence
\begin{equation*}\begin{split}((h_{(123)})_{(1)}a)((h_{(123)})_{(2)}b)
=\sum_{\sigma\rho=(123)} (h_\sigma a)(h_\rho b) = ab  \ne 
\\ ((h_{(123)})_{(2)}a)((h_{(123)})_{(1)}b) = \sum_{\sigma\rho=(123)} (h_\rho a)(h_\sigma b)=0.\end{split}
\end{equation*}
Therefore the dual base element $h_{(123)}$ in $(FS_3)^*$ satisfies the needed property and hence the universal Hopf algebra is not cocommutative.
\end{example}

Recall that Theorem~\ref{TheoremUniversalHopfOfAGradingIsJustUniversalGroup} asserts that the universal Hopf algebra of a coaction that corresponds to a group grading
is just the group algebra of the universal group of the grading.
In Proposition~\ref{PropositionDoubleNumbersUniversal} below we show that the analogous result for group actions does not hold, i.e. the universal Hopf algebra of a group action is not necessarily a group algebra.
In addition, here we demonstrate how the notion of the universal cocommutative Hopf algebra
can be used to calculate the universal Hopf algebra.

\begin{proposition}\label{PropositionDoubleNumbersUniversal}
Let $A = F[x]/(x^2)$ where $F$ is an algebraically closed field, $\ch F =0$. Consider $G$ be the cyclic group of order $2$ with the generator $c$. Define a $G$-action on $A$
by $c \bar x = -\bar x$. Then the universal Hopf algebra
of the corresponding $FG$-action $\zeta_0 \colon FG \to \End_F(A)$ equals
$(H, \zeta)$ where $H=F[y]\otimes FF^\times$
 as an algebra and a coalgebra where
the coalgebra structure on $F[y]$ is defined by
$\Delta(y)=1 \otimes y + y \otimes 1$ and $\varepsilon(y)=0$. The antipode $S$ of $H$
and the action $\zeta \colon H \to \End_F(A)$
are defined by $$S(y^k\otimes \lambda)=(-1)^k y^k\otimes \lambda^{-1}\text{ and }\zeta(y^k\otimes \lambda)\bar x= \lambda$$ for $k\in\mathbb Z_+$ and $\lambda \in F^\times$. 
In particular, $H$ is not a group algebra since it contains non-trivial primitive elements.
\end{proposition}
\begin{proof}
To this end, we first identify $\End_F(A)$ with the algebra
$M_2(F)$ of $2\times 2$ matrices by fixing the basis 
$\bar 1, \bar x$ in $A$. Then $\zeta_0(1) = \left(\begin{smallmatrix} 1 & 0 \\
0 & 1\end{smallmatrix}\right)$ and $\zeta_0(c)= \left(\begin{smallmatrix} 1 & 0 \\
0 & -1\end{smallmatrix}\right)$. Since $\zeta_0(FG)$ is the linear span of $\zeta_0(1)$
and $\zeta_0(c)$, we get that $\zeta_0(FG)$ is the subalgebra of all diagonal matrices in $M_2(F)$.
Moreover $\Aut(A) = \left\lbrace\left(\begin{smallmatrix} 1 & 0 \\
0 & \lambda \end{smallmatrix}\right) \mathbin{\bigl|} \lambda \in F^\times \right\rbrace$
and $\Der(A) = \left\lbrace\left(\begin{smallmatrix} 0 & 0 \\
0 & \lambda \end{smallmatrix}\right) \mathbin{\bigl|} \lambda \in F \right\rbrace$.
Hence $\mathcal U(\zeta_0(FG)) \cap \Aut(A)=\Aut(A) \cong F^\times$
and $\zeta_0(FG) \cap \Der(A)$ is a one dimensional Lie algebra
whose universal enveloping algebra is isomorphic to $F[y]$.
Now Theorem~\ref{TheoremUniversalCocommutative} implies
that the universal cocommutative Hopf algebra of $\zeta_0$ is indeed $(H,\zeta)$.

By Theorem~\ref{TheoremCocommUnivIsUniv}, in order to show that $(H,\zeta)$ is universal as a not necessarily cocommutative Hopf algebra,
it is sufficient to show that any other Hopf algebra action $\zeta_1 \colon H_1 \to \End_F(A)$
equivalent to $\zeta_0$, factors through a cocommutative Hopf algebra.

Suppose $\zeta_1 \colon H_1 \to \End_F(A)$ is an $H_1$-module structure on $A$
equivalent to $\zeta_0$. Then $\zeta_1(H_1)$ is the algebra of all diagonal $2\times 2$ matrices.
In particular, $\bar x$ is a common eigenvector of operators from $H_1$.
Define $\varphi \in H_1^*$ by $h\bar x = \varphi(h) \bar x$ for all $h\in H_1$. Then $\varphi \colon H_1 \to F$
is a unital algebra homomorphism and therefore a group-like element  of the Hopf algebra $H^\circ$.
At the same time
Proposition~\ref{PropositionHopfUnivModUnitality} implies
 $h\bar 1 = \varepsilon(h) \bar 1$  for all $h\in H_1$.

The powers $\varphi^k$, $k\in \mathbb Z$, of the element $\varphi$ in the group $G(H^\circ)$ of the group-like elements of $H^\circ$ are defined by the formula
  $$\varphi^k(h)=\left\lbrace\begin{array}{ccc}\varphi(h_{(1)})\dots \varphi(h_{(k)})
& \text{ for } & k \geqslant 1, \\
\varepsilon(h)  & \text{ for } & k = 0,\\
 \varphi(Sh_{(1)})\dots \varphi(Sh_{(k)})
& \text{ for } & k \leqslant -1\end{array}\right.$$
for every $h\in H_1$.
Let $$I := \bigcap\limits_{k\in\mathbb Z} \ker \bigl(\varphi^k\bigr).$$
Since all $\varphi^k \colon H_1 \to F$ are unital algebra homomorphisms,
$I$ is an ideal. Moreover, as we include in the definition of $I$ the negative powers of $\varphi$ too,
we have $SI\subseteq I$.

We claim that $I$ is a coideal too, i.e. $\Delta(I) \subseteq H_1 \otimes I + I \otimes H_1$.
 We first notice that \begin{equation}\label{EqH_1IIH_1}H_1 \otimes I + I \otimes H_1 = \bigcap\limits_{k,\ell\in\mathbb Z} \ker \bigl(\varphi^k \otimes \varphi^\ell\bigr).\end{equation} The inclusion 
 $$H_1 \otimes I + I \otimes H_1 \subseteq \bigcap\limits_{k,\ell\in\mathbb Z} \ker \bigl(\varphi^k \otimes \varphi^\ell\bigr)$$ is obvious. In order to prove the converse inclusion, we choose
 a basis $(a_\alpha)_{\alpha \in \Lambda_1 \cup \Lambda_2}$ in $H_1$ that contains a basis $(a_\alpha)_{\alpha \in \Lambda_1}$ of $I$.
Suppose $$w:=\sum_{\alpha,\beta \in \Lambda_1 \cup \Lambda_2} \gamma_{\alpha\beta}\, a_{\alpha}\otimes a_{\beta}
\in \bigcap\limits_{k,\ell\in\mathbb Z} \ker \bigl(\varphi^k \otimes \varphi^\ell\bigr)$$
where only a finite number of coefficients $\gamma_{\alpha\beta} \in F$ are nonzero. 
From the definition of $I$ we obtain $$\sum_{\alpha,\beta \in \Lambda_1 \cup \Lambda_2} \gamma_{\alpha\beta} \varphi^k (a_{\alpha}) a_{\beta} \in I \text{ for every }k\in\mathbb Z.$$
The properties of $a_\alpha$ imply that for all $\beta \in \Lambda_2$  $$\sum_{\alpha \in \Lambda_1 \cup \Lambda_2} \gamma_{\alpha\beta} \varphi^k (a_{\alpha})=0,$$
i.e. $$\sum_{\alpha \in \Lambda_1 \cup \Lambda_2} \gamma_{\alpha\beta} a_{\alpha}\in I.$$
Hence $\gamma_{\alpha\beta} = 0$ for all $\alpha,\beta \in \Lambda_2$.
Therefore $w\in H_1 \otimes I + I \otimes H_1$
and \eqref{EqH_1IIH_1} follows.
  Since for every $k,\ell \in\mathbb Z$ and $h\in I$ we have 
  $$(\varphi^k \otimes \varphi^\ell)(h_{(1)}\otimes h_{(2)})
  =\varphi^{k+\ell}(h)=0,$$ the $S$-invariant ideal $I$ is a coideal and therefore a Hopf ideal.
  
  Recall that $I \subseteq \ker\varepsilon \cap \ker\varphi$. Hence the homomorphism $\zeta_1$ factors through the Hopf algebra $H_1/I$, i.e. the following diagram is commutative:
\begin{equation*}\xymatrix{ H _1 \ar[r]^(0.45){\pi} \ar[rd]_{\zeta_1}& H_1/I \ar@{-->}[d] \\
& \End_F(A)
}
\end{equation*}
where $\pi \colon H_1 \twoheadrightarrow H_1/I$
 is the natural surjective homomorphism. 
 
  In order to show that $H_1/I$ is cocommutative, it is sufficient to prove 
  that $$h_{(1)}\otimes h_{(2)}-
 h_{(2)}\otimes h_{(1)} \in H_1 \otimes I + I \otimes H_1
\text{ for every }h\in H_1.$$ By~\eqref{EqH_1IIH_1} it is sufficient to check that
$$(\varphi^k \otimes \varphi^\ell)(h_{(1)}\otimes h_{(2)}-h_{(2)}\otimes h_{(1)})=0.$$
Indeed, $$(\varphi^k \otimes \varphi^\ell)(h_{(1)}\otimes h_{(2)}-h_{(2)}\otimes h_{(1)})=
\varphi^{k+\ell}(h) - \varphi^{k+\ell}(h) = 0.$$
Therefore $H_1/I$ is cocommutative, as desired.
\end{proof}

\subsection{Unital module structures on $F[x]/(x^2)$}\label{SubsectionDoubleNumbers}

In this section we classify all unital module structures on $F[x]/(x^2)$.
Throughout, $H_{4}$ denotes the Sweedler Hopf algebra. Recall that $H_{4}$ is generated as an algebra by two elements $c$ and $v$ subject to the relations $c^{2} = 1$, $v^{2} = 0$ and $vc = -cv$ while the coalgebra structure and the antipode are given as follows:
$$
\Delta(c)=c\otimes c,\quad \Delta(v)=c\otimes v + v\otimes 1, \quad S(c)=c,\quad S(v)=-cv.
$$

\begin{theorem}\label{TheoremDoubleNumbersClassify}
Let $\zeta \colon H \to \End_F(A)$ be a unital $H$-module structure on $A = F[x]/(x^2)$ where $H$ is a Hopf algebra, $\ch F \ne 2$.
Then $\zeta$ is equivalent to one of the following module
structures on $A$:
\begin{enumerate}
\item the action of $F$ on $A$ by the multiplication by scalars;

\item the $FG$-action, where $G=\langle c \rangle_2$, defined by $c \bar x = -\bar x$;

\item the $H_4$-action defined by
$c \bar 1 = \bar 1$, $c \bar x = -\bar x$, $v \bar 1 = 0$,
$v\bar x = \bar 1$. 
\end{enumerate}
\end{theorem}
\begin{proof}
Again, fix the basis $\bar 1$, $\bar x$ and identify $\End_F(A)$ with $M_2(F)$.
Since $\zeta$ is unital, there exist $\alpha,\beta \in H^*$
such that $\zeta(h)=\left(\begin{smallmatrix} \varepsilon(h) & \beta(h) \\
0 & \alpha(h) \end{smallmatrix}\right)$ for every $h\in H$.

Note that~(\ref{EqModCompat}) implies
\begin{equation*}\begin{split}
0 = h(\bar x^2)=(h_{(1)} \bar x)(h_{(2)} \bar x)
=(\beta(h_{(1)})\bar 1 + \alpha(h_{(1)}) \bar x)(\beta(h_{(2)})\bar 1 + \alpha(h_{(2)}) \bar x)
\\ =\beta(h_{(1)})\beta(h_{(2)})\bar 1 + (\alpha(h_{(1)})\beta(h_{(2)}) + \beta(h_{(1)})\alpha(h_{(2)}))\bar x
\end{split}\end{equation*}
and
\begin{equation}\label{EqBetaBeta}\beta(h_{(1)})\beta(h_{(2)})=0,\end{equation}
\begin{equation}\label{EqAlphaBeta}\alpha(h_{(1)})\beta(h_{(2)}) + \beta(h_{(1)})\alpha(h_{(2)})=0\end{equation}
for all $h\in H$.

We have $\dim \zeta(H)=\dim \langle \alpha, \beta, \varepsilon \rangle_F$.

If $\dim \zeta(H) = 3$, then $\zeta(H)$ is the subalgebra of all upper triangular matrices
and $\zeta$ is equivalent to (3).

Suppose $\dim \zeta(H) \leqslant 2$.
If $\alpha$ and $\varepsilon$ are linearly dependent,
then $\varepsilon\ne 0$ implies $\alpha = \gamma \varepsilon$ for some $\gamma \in F$.
Therefore, $1=\alpha(1_H) = \gamma \varepsilon(1_H)=\gamma$ and $\alpha = \varepsilon$.
By~(\ref{EqAlphaBeta}), $2\beta(h)=0$ and $\beta = 0$. Then $\zeta(H)$ is the algebra
of all scalar matrices and $\zeta$ is equivalent to (1).

Suppose $\dim \zeta(H) \leqslant 2$, but $\alpha$ and $\varepsilon$ are linearly independent.
Then $\beta = \lambda \varepsilon + \gamma \alpha$
for some $\lambda,\gamma \in F$.
Applying both sides of this equality to $1_H$,
we get $0=\lambda+\gamma$ since $\zeta(1_H)=\left(\begin{smallmatrix} 1 & 0 \\
0 & 1 \end{smallmatrix}\right)$.
Therefore $\beta= \lambda(\varepsilon-\alpha)$ and~(\ref{EqBetaBeta}) and~(\ref{EqAlphaBeta})
imply
\begin{equation}\label{EqDoubleNumbersLambdaVarepsilon}\lambda^2\bigl(\varepsilon(h)-2
\alpha(h)+ \alpha(h_{(1)})\alpha(h_{(2)})\bigr)=0,\end{equation}
\begin{equation}\label{EqDoubleNumbersLambdaAlpha}2\lambda\left(\alpha(h) - \alpha(h_{(1)})\alpha(h_{(2)})\right)=0\end{equation}
for all $h\in H$. Suppose $\lambda \ne 0$.
Since $\ch F \ne 2$, (\ref{EqDoubleNumbersLambdaAlpha}) implies
$$\alpha(h_{(1)})\alpha(h_{(2)})= \alpha(h).$$
Then from~(\ref{EqDoubleNumbersLambdaVarepsilon}) we get 
$$\alpha(h)=\varepsilon(h)$$
which contradicts with the linear independence of
 $\alpha$ and $\varepsilon$.
Therefore, $\lambda = 0$ and $\beta = 0$.
Thus $\zeta(H)$ is the subalgebra of all diagonal matrices
and $\zeta$ is equivalent to (2).
\end{proof}

\subsection{Polynomial $H$-identities}\label{SubsectionHPI}

In this section we will show that the module algebras classified in the previous section all satisfy the analog of Amitsur's conjecture for
polynomial $H$-identities.

We begin with a quick introduction into the theory of polynomial $H$-identities 
(see also~\cite{BahturinLinchenko, BereleHopf, ASGordienko13}).
All algebras in this section are associative, but not necessarily unital.

We denote by $F\langle X \rangle$ the \textit{free non-unital associative algebra} on the
countable set $X=\lbrace x_1, x_2, x_3, \ldots\rbrace$, i.e. the algebra of all polynomials
without a constant term in non-commuting variables $x_1, x_2, x_3, \ldots$ with coefficients from $F$.
Then $F\langle X \rangle = \bigoplus_{n=1}^\infty F\langle X \rangle^{(n)}$
where $F\langle X \rangle^{(n)}$ is the linear span of monomials of total degree $n$.

Let $H$ be a Hopf algebra. Then the \textit{free non-unital associative $H$-module
algebra} on $X$ is $$F\langle X|H \rangle := \bigoplus_{n=1}^\infty F\langle X \rangle^{(n)} \otimes H^{{}\otimes n}$$ where $$h \cdot v\otimes h_1 \otimes \dots \otimes h_n := v
\otimes h_{(1)}h_1 \otimes \dots \otimes h_{(n)}h_n \text{ for } v\in F\langle X \rangle^{(n)}
\text{ and } h_i\in H$$
and $(v_1 \otimes w_1)(v_2 \otimes w_2):=v_1 v_2 \otimes w_1 \otimes w_2$
for $v_1\in F\langle X \rangle^{(k)}$, $v_2\in F\langle X \rangle^{(\ell)}$,
$w_1 \in H^{{}\otimes k}$, $w_2 \in H^{{}\otimes \ell}$, $k,\ell\in\mathbb N$.
We use the notation $x_{i_1}^{h_1}\ldots x_{i_n}^{h_n} := x_{i_1}\ldots x_{i_n} 
\otimes h_1 \otimes h_2 \otimes \dots \otimes h_n$.

Let $(h_{\alpha})_{\alpha\in \Lambda}$ be a basis in $H$. Then $F\langle X|H \rangle$ is 
isomorphic as an algebra to the free non-unital associative algebra on the set 
$\lbrace x_n^{h_\alpha} \mid n\in\mathbb N,\  \alpha\in \Lambda\rbrace$.

Now we identify $X$ with the subset $\lbrace x_n^1 \mid n\in\mathbb N\rbrace \subseteq F\langle X|H \rangle$. Denote by $\iota \colon X \to F\langle X|H \rangle$ the corresponding embedding.
Then $F\langle X|H \rangle$ satisfies the following universal property:
for any map $\varphi \colon X \to B$ where $B$ is an $H$-module algebra
there exists a unique homomorphism of algebras and $H$-modules $\bar \varphi \colon F\langle X | H \rangle \to B$ such that the following diagram commutes:

\begin{equation*}
\xymatrix{ X  \ar[r]^(0.4){\iota} \ar[rd]_{\varphi}& F\langle X|H \rangle \ar[d]^{\bar\varphi} \\
& B}
\end{equation*}
On the generators $\bar\varphi$ is defined as follows:
$$\bar\varphi\left(x_{i_1}^{h_1}\ldots x_{i_n}^{h_n}\right)=(h_1\varphi(x_{i_1}))\ldots 
(h_n\varphi(x_{i_n})).$$
\begin{remark}
Note that in a similar way one can define the free $H$-module algebra on any set. Then $F\langle - | H \rangle$ becomes the left adjoint functor to the forgetful functor from the category of not necessarily
unital associative $H$-module algebras to the category of sets.
\end{remark}

The elements of $F\langle X|H \rangle$ are called \textit{$H$-polynomials}.
For a given $H$-module algebra $A$ the intersection $\Id^H(A)$ of the kernels
of all possible $H$-module algebra homomorphisms $F\langle X|H \rangle \to A$ is called
the set of \textit{polynomial $H$-identities}. Taking into account the universal property
of $F\langle X|H \rangle$, it is easy to see that $\Id^H(A)$ consists of all $H$-polynomials
that vanish under all substitutions of elements of $A$ for their variables.
In addition, $\Id^H(A)$ is an $H$-invariant ideal of $F\langle X|H \rangle$ invariant under
all endomorphisms of $F\langle X|H \rangle$ as an $H$-module algebra.

The algebra $F\langle X|H \rangle / \Id^H(A)$ is called the \textit{relatively free $H$-module algebra of the variety of $H$-module algebras generated by $A$}. Indeed, it is easy to see that $F\langle X|H \rangle / \Id^H(A)$ satisfies
the same universal property as $F\langle X|H \rangle$ except that we consider only $H$-module algebras $B$ satisfying $\Id^H(A) \subseteq \Id^H(B)$.

In some polynomial $H$-identities below we use other variables such as $x,y,\ldots$,
always assuming that in fact these variables coincide with some of the variables $x_1, x_2, \ldots$.

\begin{example}\label{ExampleIdH}
Consider the $\mathbb Z/2\mathbb Z$-grading
on $M_2(F)$ defined by $M_2(F)^{(\bar 0)} = \left\lbrace\left(\begin{smallmatrix}
* & 0 \\
0 & *
\end{smallmatrix}
 \right)\right\rbrace$
and 
$M_2(F)^{(\bar 1)} = \left\lbrace\left(\begin{smallmatrix}
0 & * \\
* & 0
\end{smallmatrix}
 \right)\right\rbrace$ and the corresponding $(F(\mathbb Z/2\mathbb Z))^*$-action.
 Then $$x^{h_0}y^{h_0}-y^{h_0}x^{h_0}\in \Id^{(F(\mathbb Z/2\mathbb Z))^*}(M_2(F))$$
 where $h_0\in (F(\mathbb Z/2\mathbb Z))^*$ is defined by $h_0(\bar 0) = 1$ and $h_0(\bar 1) = 0$.
\end{example}

The classification of all varieties of algebras with respect to their ideals
of polynomial identities seems to be a wild problem. This is the reason why
numeric characteristics of polynomial identities are studied. 
One of the most important
numeric characteristics is the codimension sequence.

Let $$P_n^H := \langle x_{\sigma(1)}^{h_1} x_{\sigma(2)}^{h_2}\dots x_{\sigma(n)}^{h_n} \mid \sigma \in S_n,\ h_i\in H\rangle_H \subset F\langle X|H \rangle$$
where $S_n$ is the $n$th symmetric group and
$n\in\mathbb N$. The elements of $P_n^H$ are called \textit{multilinear $H$-polynomials}
and the elements of $P_n^H \cap \Id^H(A)$ are called \textit{multilinear polynomial
$H$-identities} of $A$.

Using the linearization process~\cite[Section~1.3]{ZaiGia}, it is not difficult to see that
over a field of characteristic $0$ every polynomial $H$-identity of an $H$-module algebra $A$ is equivalent
to a finite set of multilinear polynomial $H$-identities of $A$. Therefore the spaces
$P_n^H \cap \Id^H(A)$, where $n\in\mathbb N$, contain all the information
of polynomial $H$-identities of $A$. For a given $H$-module algebra $A$ the number $c_n^H(A):=\dim\left(\frac{P_n^H}{P_n^H\cap\, \Id^H(A)}\right)$, $n\in\mathbb N$, is called the \textit{$n$th codimension of polynomial $H$-identities}
of $A$. 

In the case when $H=F$ we obtain ordinary polynomial identities and their codimensions.
In the case when $H=(FG)^*$ for some finite group $G$, every $H$-module algebra $A$ is just
a $G$-graded algebra. Here one can introduce graded polynomial identities
and their codimensions~$c_n^{G\text{-gr}}(A)$, however it turns out that $c_n^{G\text{-gr}}(A)
=c_n^{(FG)^*}(A)$ for every $n\in\mathbb N$~\cite[Lemma~1]{ASGordienko3}.

In the 1980's, a conjecture concerning the asymptotic behaviour of codimensions of ordinary polynomial
identities was made by S.\,A.~Amitsur~\cite[Conjecture~6.1.3]{ZaiGia}. This conjecture was proved in 1999 by A.\,Giambruno and M.\,V.~Zaicev~\cite[Theorem~6.5.2]{ZaiGia}. For polynomial $H$-identities
the analog of Amitsur's conjecture can be formulated in the following form, which
belongs to Yu.\,A.~Bahturin:

\begin{conjecture}\label{ConjectureAmitsurBahturin} Let $A$ be a finite dimensional associative $H$-module algebra
for a Hopf algebra $H$ over a field of characteristic $0$. Then there exists
an integer $\PIexp^H(A):=\lim\limits_{n\to\infty}
 \sqrt[n]{c^H_n(A)}$.
\end{conjecture}

Conjecture~\ref{ConjectureAmitsurBahturin} was proved in~\cite[Theorem~3]{ASGordienko3} for $H$ finite dimensional semisimple (this result was later generalized by Ya.~Karasik~\cite{Karasik} for the case when $A$ is a not necessarily finite dimensional PI-algebra) and in~\cite[Theorem~1]{ASGordienko8} for $H$-module algebras $A$ such that the Jacobson radical $J(A)$ is an $H$-submodule (the requirement that $A/J(A)$ is a direct sum of $H$-simple algebras is satisfied by~\cite[Theorem 1.1]{SkryabinHdecomp}, \cite[Lemma 4.2]{SkryabinVanOystaeyen}).
In the case when $J(A)$ is not an $H$-submodule, the conjecture was proved only for Hopf algebras $H$
that are iterated Ore extensions of finite dimensional semisimple Hopf algebras~\cite[Corollary~7.4]{ASGordienko13}.

The notion of equivalence of module structures reduces drastically the number
 of cases one has to consider in order to prove Conjecture~\ref{ConjectureAmitsurBahturin}:

\begin{lemma}\label{LemmaHEquivCodimTheSame}
Let $\zeta_1 \colon H_1 \to \End_F(A_1)$ and $\zeta_2 \colon H_2 \to \End_F(A_2)$
be equivalent module structures on algebras $A_1$ and $A_2$.
Then there exists an algebra isomorphism $F\langle X | H_1\rangle/\Id^{H_1}(A_1) 
\mathrel{\widetilde\to} F\langle X | H_2\rangle/\Id^{H_2}(A_2)$ that
for every $n\in \mathbb N$ restricts
to an isomorphism $$\frac{P_n^{H_1}}{P_n^{H_1}\cap\,\Id^{H_1}(A_1)}\mathrel{\widetilde\to} \frac{P_n^{H_2}}{P_n^{H_2}\cap\,\Id^{H_2}(A_2)}.$$
In particular, $c_n^{H_1}(A_1) = c_n^{H_2}(A_2)$.
\end{lemma}
\begin{proof} Let $\varphi \colon A_1 \mathrel{\widetilde\to} A_2$ be an equivalence
of $H_1$- and $H_2$-module structures and let $\tilde\varphi \colon \End_F(A_1) \mathrel{\widetilde\to} \End_F(A_2)$ be the corresponding isomorphism of algebras of linear operators.
We have $\tilde\varphi(\zeta_1(H_1))=\zeta_2(H_2)$. Hence there exist $F$-linear maps $\xi \colon H_1 \to H_2$ and $\theta \colon H_2 \to H_1$ (which are not necessary homomorphisms) such that the following diagram commutes in both ways:
\begin{equation*}
\xymatrix{ H_1  \ar@<0.5ex>[r]^\xi \ar[d]_{\zeta_1}& H_2 \ar@<0.5ex>[l]^\theta \ar[d]^{\zeta_2} \\
\End_F(A_1) \ar[r]^{\tilde\varphi} & \End_F(A_2)}
\end{equation*}

Then \begin{equation}\label{EqXiActingCodimTheSame}\xi(h)\varphi(a)=\varphi(ha)\text{ for every }h\in H_1\text{ and }a\in A_1\end{equation}
and  \begin{equation}\label{EqZetaActingCodimTheSame}\varphi(\theta(h)a)=h\varphi(a)\text{ for every }h\in H_2\text{ and }a\in A_1.\end{equation}

Now we define the algebra homomorphisms $\tilde\xi \colon F\langle X | H_1\rangle
\to F\langle X | H_2\rangle$ and $\tilde\theta \colon F\langle X | H_2\rangle
\to F\langle X | H_1\rangle$
by $\tilde \xi\left(x_k^{h}\right):=x_k^{\xi(h)}$
for $h\in H_1$, $k\in\mathbb N$
and $\tilde \theta\left(x_k^h\right):=x_k^{\theta(h)}$
for $h\in H_2$, $k\in\mathbb N$. Equations~\eqref{EqXiActingCodimTheSame}
and~\eqref{EqZetaActingCodimTheSame} imply that 
$\tilde\xi\left( \Id^{H_1}(A_1) \right) \subseteq \Id^{H_2}(A_2)$
and 
$\tilde\theta\left( \Id^{H_2}(A_2) \right) \subseteq \Id^{H_1}(A_1)$.
Hence $\tilde\xi$ and $\tilde\theta$ induce homomorphisms
$\bar\xi \colon F\langle X | H_1\rangle/\Id^{H_1}(A_1)
\to F\langle X | H_2\rangle/\Id^{H_2}(A_2)$ and $\bar\theta \colon F\langle X | H_2\rangle/\Id^{H_2}(A_2)
\to F\langle X | H_1\rangle/\Id^{H_1}(A_1)$.
Note that~\eqref{EqXiActingCodimTheSame}
and~\eqref{EqZetaActingCodimTheSame} also imply $\theta(\xi(h))a=ha$ for all $a\in A_1$
and $h\in H_1$ and $\xi(\theta(h))a=ha$ for all $a\in A_2$
and $h\in H_2$. Hence $x^h - x^{\theta(\xi(h))}  \in \Id^{H_1}(A_1)$ and $x^h - x^{\xi(\theta(h))} \in \Id^{H_2}(A_2)$. As a consequence, $\bar\theta\bar\xi = \id_{F\langle X | H_1\rangle/\Id^{H_1}(A_1)}$
and $\bar\xi\bar\theta = \id_{F\langle X | H_2\rangle/\Id^{H_2}(A_2)}$, i.e. we get the desired isomorphism.

Comparing degrees of $H$-polynomials, we get 
$\bar\xi \left(\frac{P_n^{H_1}}{P_n^{H_1}\cap\,\Id^{H_1}(A_1)}\right) \subseteq \frac{P_n^{H_2}}{P_n^{H_2}\cap\,\Id^{H_2}(A_2)}$
and $\bar\theta \left(\frac{P_n^{H_2}}{P_n^{H_2}\cap\,\Id^{H_2}(A_2)}\right) \subseteq \frac{P_n^{H_1}}{P_n^{H_1}\cap\,\Id^{H_1}(A_1)}$ for every $n\in\mathbb N$.
\end{proof}

Now we can prove Conjecture~\ref{ConjectureAmitsurBahturin} for any unital Hopf algebra action
on $F[x]/(x^2)$.

\begin{theorem}\label{TheoremDoubleNumbersAmitsurPIexpH}
Suppose $F[x]/(x^2)$ is a unital $H$-module algebra for some Hopf algebra $H$ over a field $F$ of characteristic $0$. Denote by $d$ the dimension of the maximal $H$-invariant nilpotent ideal in $F[x]/(x^2)$.
Then there exist $C_1,C_2 > 0$ and $r_1,r_2 \in \mathbb R$ such that \begin{equation}\label{EquationHAmitsur}C_1 n^{r_1} (2-d)^n \leqslant c^H_n(F[x]/(x^2)) \leqslant C_2 n^{r_2} (2-d)^n\end{equation}
for all $n\in\mathbb N$.

In particular, the analog of Amitsur's conjecture holds for polynomial $H$-identities of $F[x]/(x^2)$. 
\end{theorem}
\begin{proof} Lemma~\ref{LemmaHEquivCodimTheSame} implies that it is sufficient to prove~\eqref{EquationHAmitsur} for module structures described in Theorem~\ref{TheoremDoubleNumbersClassify}. In the first two cases the Jacobson radical of $F[x]/(x^2)$, which equals $F\bar x$, is $H$-invariant, i.e. $d=1$ and~\eqref{EquationHAmitsur} follows from
\cite[Theorem~1]{ASGordienko8}. In the last case $F[x]/(x^2)$ is an $H_4$-simple algebra (see~\cite{ASGordienko11, ASGordienko12}),
i.e. $d=0$ and~\eqref{EquationHAmitsur} follows from~\cite[Theorem~7.1, Corollary~7.4]{ASGordienko13}.
\end{proof}

\section*{Acknowledgements}

The authors are indebted to the referee for the many valuable comments and suggestions which led to a substantial improvement of the paper.
This work started while the second author was an FWO postdoctoral fellow at
Vrije Universiteit Brussel, whose faculty and staff he would like to thank for hospitality and many useful discussions. In addition, the authors are grateful to Yorck Sommerh\"auser for helpful discussions.

\end{document}